\documentclass[reqno,10pt]{amsart}
\usepackage{amscd,amssymb,amsmath,color,url,stmaryrd}
\usepackage{latexsym}
\usepackage[all]{xy}
\usepackage{defs}
\usepackage[colorlinks=true]{hyperref}
\usepackage{graphicx}
\usepackage{tikz}
\usetikzlibrary{patterns}
\SelectTips{cm}{10}
\everyxy={<2.5em,0em>:}
\usepackage[text={160mm,230mm},centering,headsep=2mm,footskip=5mm]{geometry}

\def\sqrtNJ{{\sqrt[N]{\cJ}}}
\def\sqrtNV{{\sqrt[N]{V}}}

\def\NN{{\bbN}}
\def\GG{{\bbG}}
\def\QQ{{\bbQ}}
\def\RR{{\bbR}}

\def\cProj{{\mathcal{P}roj}}

\def\nor{{\rm nor}}
\def\ord{{\rm ord}}
\def\Mord{{\underline{\rm ORD}}}
\def\Mordone{{\underline{\rm ORD}}_{>1}}
\def\mord{{\underline{\rm ord}}}
\def\ford{{\underline{\rm ford}}}
\def\hatcO{{\widehat\cO}}
\def\hatcI{{\widehat\cI}}
\def\hatcJ{{\widehat\cJ}}

\def\.{{,\dots,}}

\def\cJ{{\calJ}}
\def\cT{{\calT}}
\def\cR{{\calR}}

\def\cI{{\calI}}
\def\cO{{\calO}}
\def\cC{{\calC}}
\def\cF{{\calF}}

\def\cD{{\calD}}

\def\ocJ{{\overline{\cJ}}}

\def\cN{{\calN}}
\def\cQ{{\calQ}}

\begin{document}
\title{Dream resolution and principalization II: excellent schemes}
\author{Michael Temkin}

\thanks{This research is supported by ERC Consolidator Grant 770922 - BirNonArchGeom and  BSF grant 2018193. The final part of this research work was done during my visits at IHES and MPIM-Bonn. The author is grateful to D. Abramovich for careful reading the manuscript and making numerous valuable suggestions and to D. Abramovich, A. Belotto, J. W{\l}odarczyk and M. McQuillan for valuable discussions.}

\address{Einstein Institute of Mathematics\\
               The Hebrew University of Jerusalem\\
                Edmond J. Safra Campus, Giv'at Ram, Jerusalem, 91904, Israel}
\email{michael.temkin@mail.huji.ac.il}

\maketitle
\begin{abstract}
This is the second paper in a project on dream (or memoryless) principalization and resolution methods. It extends this theory from the case of schemes with enough derivations, which was established in \cite{dream_derivations}, to general excellent schemes of characteristic zero. So, similarly to McQuillan's approach developed in \cite{McQuillan}, the approach of \cite{ATW-weighted} is now extended to the generality of all excellent schemes of characteristic zero. In addition, we precisely describe the set of invariants of canonical centers and establish the resolution in the non-embedded form, where one applies simple (stack-theoretical) modifications along subschemes of a special form that we call tubes. In the regular case these are precisely the subschemes corresponding to canonical centers.
\end{abstract}

\section{Introduction}

\subsection{Background and motivation}

\subsubsection{Motivation and main results}
By a {\em dream resolution or principalization algorithm} we mean those that have no memory -- at each step an invariant of the singularity and the modification center are computed independently, and the invariant drops after blowing up the center. Such algorithms were discovered recently and they exist in various contexts, including varieties, log varieties, foliated varieties, etc. This paper is a sequel of \cite{dream_derivations}, where a dream principalization on regular schemes (or stacks) of characteristic zero with enough derivations was constructed and dream resolution of schemes locally embeddable into such regular schemes was deduced. In fact, it was checked that the algorithm of \cite{ATW-weighted} for varieties extends to this generality, though some arguments were modified (and hopefully simplified) so that it will be easier to transport them to other settings in future research. The goal of this paper is to extend the same algorithms to their maximal possible generality -- those of quasi-excellent schemes of characteristic zero (in fact, the qe schemes we will work with are excellent, see Remark~\ref{catenaryrem}(ii)).

First, we will show that the {\em dream principalization} indeed extends to arbitrary regular excellent schemes. The precise definition is given in \S\ref{princsec}, but loosely speaking at each step of the dream principalization one finds the maximal $\cI$-admissible center $\cJ$, called the {\em canonical center}, blows up, stack-theoretically, its appropriate power $\cJ^{1/N}$ so that the blown up stack stays regular, transforms $\cI$ by dividing its pullback by the pullback of $\cJ$, and achieves in this way that the multiorder of the next canonical center is strictly smaller.

\begin{theor}\label{principalizationth}
Let $X$ be a noetherian regular excellent stack of characteristic zero and $\cI$ an ideal on $X$, then:

(i) The dream principalization sequence of $\cI$ is finite and ends with $X_n$ such that $\cI_n=\cO_{X_n}$ is trivial.

(ii) The sequence $X_n\to\dots\to X_0=X$ is trivial over the complement of $V(\cI)$. In particular, it is birational over each component of $X$ not lying in $V(\cI)$.

(iii) The sequence provides a principalization of $\cI$: the ideal $\cI\cO_{X_n}$ is invertible.

(iv) The sequence is functorial with respect to any regular morphism $f\:Y\to X$: the dream principalization of $\cI\cO_Y$ is obtained from the dream principalization of $\cI$ by pulling it back to a sequence of normalized root blowings up of $Y$ and removing all blowings ups with empty centers.
\end{theor}

This implies the embedded dream resolution, see Theorem~\ref{embeddedth}, but our second major result is the more general non-embedded version, which applies to arbitrary formally-equidimensional excellent schemes. We refer to \S\ref{dreamressec} for precise definitions, but loosely speaking we introduce certain nilpotent thickenings of regular subschemes, called tubes, and associate to them appropriate Rees alebras and (stack-theoretic) blowings up. Then the dream resolution at each step just finds a tube $V$ of maximal width and blows up an appropriate Rees algebra $\cR_{\sqrt[N]{V}}$, where $N$ is chosen analogously to the principalization.

\begin{theor}\label{resolutionth}
Let $Z$ be a noetherian locally equidimensional reduced excellent stack of characteristic zero, then:

(i) The dream resolution sequence of $Z$ is finite and ends with a regular $Z_n$.

(ii) The sequence $Z_n\to\dots\to Z_0=Z$ is trivial over the regular locus of $Z$.

(iii) The induced sequence $Z_n\to\dots\to Z_0=Z$ is a resolution of $Z$ by a sequence of stack-theoretic modifications.

(iv) The sequence is functorial with respect to any regular morphism $f\:X'\to X$: the dream non-embedded resolution of $Z'=Z\times_XX'$ in $X'$ is obtained from the dream non-embedded resolution of $Z$ in $X$ by pulling it back to a sequence of normalized root blowings up of $X'$ and removing all blowings ups with empty centers.
\end{theor}

Loosely speaking, in both cases we show that the formal centers produced by the algorithms from \cite{dream_derivations} glue together to global centers under the excellence assumption, and blowing them up provides the desired algorithms in general.

\subsubsection{Comparison to the classical methods}
So far, an extension of classical resolution algorithms to arbitrary qe schemes (without derivations) was always done using Hironaka's localization or induction on codimension, e.g. see \cite[\S4]{Temkin-qe} and \cite[\S3.4]{Temkin-embedded}. This method modifies the algorithm via an induction on codimension procedure when descending the method from formal completions to arbitrary qe schemes. For functoriality reasons, this must be done even in the case of varieties, and one obtains a much more complicated (and long) method, than the method one starts with.

It is anticipated by some experts that the original algorithms and invariants (without localization) can be extended to the general qe case, but this is still an open question. The main problem is that large blocks of the algorithm run on a fixed maximal contact (and its transforms), which is not canonical and it is not clear why any of the maximal contacts should descend from the formal level. Descent of the first center (sometimes called the year zero center) works fine, as follows from our results, but this is not enough to descend the whole classical algorithm.

The main novelty in the dream algorithms is that they do not use history, so it is enough to establish just the following two steps:
\begin{itemize}
\item[(i)] Algebraize a single formal center, establishing a descent of the method from local formal schemes to local qe schemes.
\item[(ii)] Spread out a local center to a neighborhood in a qe scheme.
\end{itemize}
In this paper we develop techniques to do both tasks. It might be the case, that these techniques can also be used to extend the classical algorithm to qe schemes, but this is certainly not straightforward, would probably require additional ideas and is not studied in the paper.

\subsubsection{Comparison to the method of McQuillan}
The only basic algorithms which were extended to arbitrary qe schemes unchanged are the dream principalization and non-embedded resolution of McQuillan, see \cite{McQuillan}. Of course, these are the same algorithms as were independently discovered in \cite{ATW-weighted}, though the descriptions and proofs are very different. Thus, to some extent, our main results reprove those of McQuillan. Note that the formal descent argument in \cite{McQuillan} is incomplete, and an erratum was added later. The argument in the erratum is quite analogous to the one we use in the algebraization step.

Of course, there are some other similarities of the methods, and it seems that Tschirnhauss coordinates introduced in this paper as an alternative to the more standard iterative maximal contact coordinates are close in spirit to the method of weighted filtrations of McQuillan. While the maximal contact coordinates, when exist, provide quite short and elegant arguments (e.g. as worked out in \cite{dream_derivations}), McQuillan's method of weighted filtrations or the method of Tschirnhauss coordinates seem advantageous when extending the results to the generality of arbitrary qe schemes without enough derivations. It also seems that Tschirnhauss coordinates, which are used in this paper to spread out canonical centers can also be used to provide an explicit description of weighted filtrations.

Finally, working out resolution in the situation when embedding into a regular scheme exists only formally locally required us to introduce a new technique of tubes, which somewhat refines weighted centers. In doing so we had to provide a precise description of the set $\Mord$ of multiorders (or invariants) an ideal may have.

\subsection{The methods}
It was a surprise for me that comparing to \cite{dream_derivations} the current paper required to introduce quite a few new ideas and techniques. For this reason this section is written in a slightly unusual format: not only do I outline the methods, but also try to explain which problems they allow to bypass, and why more naive and natural attempts are insufficient -- at least, I did not find a way to get a simple proof along those lines.

\subsubsection{Descent and spreading out}
The basic strategy is very simple: formally locally there are enough derivations, so the dream principalization and resolution on the formal level were constructed in \cite{dream_derivations}. The formal centers are canonical, so one should prove that they come from local centers and expand the local centers to canonical centers in some neighborhoods. Then combining these two steps into a single proof is done by a standard noetherian induction. Referring to Section~\ref{GNsec}, the algebraization (or formal descent) must use the G-property and the extension must use the N-property. Surprisingly, both steps are far from being straightforward as we explain below. %In the case of resolution we will have to overcome an additional difficult

\subsubsection{Algebraization}
Until \S\ref{introtubes} we discuss principalization of an ideal $\cI$ on a regular excellent scheme $X$. The noetherian property, hence regularity, is not preserved by base changes: $\hatcO_x\otimes_{\cO_x}\hatcO_x$ is usually non-noetherian. Therefore usual flat descent and functoriality of the algorithm with respect to regular morphisms cannot be used. Thus, we have to show by hand that the canonical center $\hatcJ_x$ on $\hatX_x$ descends to $X_x$, see Lemma~\ref{formaldescent}: by noetherian induction we can assume that the restriction of $\cI$ onto the punctured localization $U=X_x\setminus \{x\}$ possesses a canonical center $\cJ_U$ whose pullback to $\hatU=\hatX_x\setminus\{x\}$ is the canonical center $\hatcJ_U$ of the pullback $\cI\cO_\hatU$, and the task is to prove that the canonical center $\hatcJ_x$ of $\hatcI_x$ comes from a center $\cJ_x$ on $X_x$, which is then easily seen to be the canonical center of $\cI_x$. Raising to a power one reduces to the case when $\hatcJ_x$ is an actual ideal. If $V(\hatcJ_x)=\{x\}$, then the center is an open ideal and hence comes from an ideal on $X_x$. Otherwise, $\hatcJ_x=\hatcJ_U\cap\hatcO_x$ and we simply set $\cJ_x=\cJ_U\cap\cO_x$. In geometrical terms, the center is the schematic closure of its restriction onto $U$. In this way we have already succeeded to descend the formal center to a local ideal $\cJ_x$, but it remains to show that the latter is a center. Surprisingly, this step is the most subtle one. It is based on Lemma~\ref{functorcenter} and uses that by the G-property the completion $\cO_x\to\hatcO_x$ is regular.

\subsubsection{Spreading out}
The local canonical center $\cJ_x$ constructed by formal descent extends to a center $\cJ_U$ on neighborhood $U$ of $x$, yielding after shrinking $U$ a candidate for being the canonical center of $\cI$ in a neighborhood of $x$. However, it is not easy to check that $\cJ_U$ is the canonical center of $\cI_U$ at a general enough point $y\in U$ which specializes $x$. It is tempting to pass to $\hatX_y$ and use the theory of derivations there, but the problem is that there are no derivations on $X$, so derivations on $\hatX_x$ and $\hatX_y$ are completely unrelated. Alternatively, one could attempt to show that the multiorder of $\cI$ is upper semicontinuous, and hence is the same at $x$ and any of its specializations once we shrink $U$ enough. Using the N-property it is easy to see that the {\em order} is upper semicontinuous, but moving this further involves induction with maximal contacts and coefficients ideals, and it is unclear how to relate formal maximal contacts and coefficients ideals at $x$ and its specializations in a small neighborhood.

The only way we could find to bypass these obstacles is to give up on the tools available formally locally -- derivations and resulting maximal contacts and coefficients ideals, and encode the fact that $\cJ_x$ is the canonical center of $\cI_x$ in terms of parameters we call {\em Tschirnhaus coordinates}, see \S\ref{Tsec}. Loosely speaking, these are approximations of iterated maximal contacts up to $\cJ^{1+\veps}$. The property is expressed in elementary terms and does not involve derivations. In particular, being Tschirnhauss is an open condition, so the extension is on the nose. In addition, given a local canonical center $\cJ_x=\cJ(\cI_x)$ it is easy to construct such coordinates at $x$. So, they do the main job in our method.

\begin{rem}
One may wonder if, moreover, using Tshicrnhaus coordinates one can reprove the results of \cite{dream_derivations} in a simpler way and in the generality of arbitrary excellent regular schemes. I do not expect that this is the case because one should somehow exploit the G-property, and it is not encoded in such elementary terms. It seems that Tschirnhaus coordinates are good to detect the canonicity once the center is already provided by other methods, but it is difficult to construct them and the center from scratch. This is illustrated by Example~\ref{nocanonical}, in which $X$ is a non-G local scheme and no canonical center exists, but one can construct an infinite sequence of approximations of the formal canonical center. In fact, even to prove that Tschirnhaus coordinates detect canonicity we prefer to pass to formal completion and use the theory of centers and maximal contacts from \cite{dream_derivations} -- at least, this seems to be the fastest route, given the results of \cite{dream_derivations}.
\end{rem}

\subsubsection{Tubes}\label{introtubes}
In the case of the resolution algorithm formal descent works quite similarly, but there is one more complication: the resolution algorithm is independent of an embedding into a regular scheme, but it uses such an embedding as an input. In general, such embeddings exist only after formal completion, and this causes an additional difficulty with spreading out -- we can extend local blowing up at $x$ to a neighborhood, but it is unclear why for any specialization $y\prec x$ in a small enough neighborhood it should coincide with the blow up induced from the formal completion at $y$. It seems that, once again, one needs to provide a constructible local object at $x$ which controls the blowing up and detects its canonicity. We solve this by encoding resolution in terms of the scheme $Z$ we want to resolve only, and without any referencing to an embedding $Z\into X$ into a regular scheme. Moreover, we have to develop a technique which applies to a non-normal $Z$ because even if we start with a normal one, the dream algorithm can produce non-normal schemes in the process of improving the multiorder via taking the strict transforms of $Z$ in blowings up of $X$.

This forces us to switch from $\cI$-admissible centers $\cJ$ on $X$ to their roundings $(\cJ)$ (see \cite[\S2.2.1]{dream_derivations}) which correspond to usual closed subschemes of $Z=V_X(\cI)$. These closed subschemes are characterized as certain nilpotent thickenings of regular schemes and are called {\em tubes} in the paper. We expect that this notion will be useful beyond the goals of the paper. To introduce tubes we precisely classify centers that show up as canonical centers of their roundings and describe the set $\Mord$ of invariants $(d_1\.d_n)$ that can show up in Hironaka's (and all subsequent) algorithms. The correct condition is that for any $i\le n$ there exists $a_1\.a_i\in\NN$ such that $\frac{a_1}{d_1}+\dots+\frac{a_i}{d_i}=1$, and it provides a much better control on the denominators of $d_1\.d_n$ than one classically gets.

Using the language of tubes we can extend the embedded resolution theory to the non-embedded situation: any formally equidimensional qe (hence automatically excellent) scheme $Z$ possesses a unique canonical tube of maximal invariant and blowing up an explicit Rees algebra associated with it one decreases the invariant, see Theorem~\ref{canonicaltube}. The non-embedded resolution follows, see Theorem~\ref{resolutionth}. Once tubes have been introduced, the arguments are quite similar to the case of principalization discussed earlier -- we use algebraization and Tschirnhaus coordinates to spread out, and in the formal situation we just embed into a regular formal variety reducing to the already established case of canonical centers. In fact, we use this method of reducing claims to the case with enough derivations even to study basic properties of tubes, such as invariants and coordinate changes.

\subsubsection{Overview of the paper}
Section \ref{princsec} extends principalization to excellent schemes. We start with studying the precise set $\Mord$ of multirders of canonical centers in \S\ref{Mordsec}. The centers whose multiorders lie in $\Mord$ are called {\em integral} in the paper, and we prove in Theorem~\ref{maxcenterth} that these are the centers which are the canonical centers of their roundings. Next, we recall the definition of qe schemes and show on examples how existence of the canonical center fails without the G and N conditions. Then we study functoriality of roundings and centers in \S\ref{algsec} and use these results to establish algebraization of canonical centers. Spreading out is done in \S\ref{extcanonicalsec} and is based on introducing Tschirnhaus coordinates and comparing them to maximal contacts on the formal level, see Theorem~\ref{tschirnhausexist}. Finally, everything done in \S\ref{princsec} is combined together in Theorem~\ref{canonicalcenter} stating that each ideal on a regular excellent scheme possesses a canonical stratification and this stratification is compatible with arbitrary regular morphisms. The principalization and embedded resolution follow easily, see Theorems~\ref{principalizationth} and \ref{embeddedth}.

The non-embedded resolution is established in \S\ref{nonembsec}. First, we introduce abstract tubes in \S\ref{tubesec} and study their basic properties in \S\ref{pressec}. In particular, we show that tubes in a regular scheme correspond to integral centers, and prove that being a tube descends with respect to regular morphisms, see Theorem~\ref{tubedescent}. This result is surprisingly delicate and (at least) to simplify the argument we use Popescu's theorem in its proof. After that we define canonical tubes analogously to canonical centers and the rest of \S\ref{nonembsec} is devoted to proving that canonical tubes exist, see Theorem~\ref{canonicaltube}. The non-embedded resolution follows immediately, see Theorem~\ref{resolutionth}. The strategy of constructing canonical tubes is similar to the proof of the principalization theorem: first we establish the formal case using embedding into a regular formal variety and existence of canonical centers, then we  show that a formal canonical tube descends to a canonical local tube using Theorem~\ref{tubedescent}, and finally we spread out a local canonical tube using a tubular version of Tschirnhaus coordinates.

\subsection{Conventiones}
We keep conventiones from \cite[\S1.3]{dream_derivations} and use notation introduced in the cited paper. In particular, all schemes (and stacks) are assumed to be noetherian of characteristic zero and we only mention this assumption in the formulations of the main results, underlines are used to denote tuples, and formal completions of schemes and ideals at a point are denoted $\hatX_x=\Spec(\hatcO_x)$ and $\hatcI_x=\cI\hatcO_x$. Similarly, localizations are denoted $\cI_x=\cI\cO_X$ and restrictions onto an open $U\subseteq X$ are denoted $\cI_U=\cI|_U$, $\cJ_U=\cJ|_U$, etc.

The paper is a direct continuation of \cite{dream_derivations}, so we will freely use the notions of $\QQ$-ideals, roundings, centers and canonical centers of ideals defined in \cite[\S2.2.1, \S2.3.1 and \S3.3.1]{dream_derivations}. Finally, we will often have to work with the weight vectors $\uw=(w_1\.w_n)$, where $w_i=\frac{1}{d_i}$, so we will use the slightly unusual notation $\uw=\ud^{-1}$.

\section{Principalization on excellent schemes}\label{princsec}
The main result of \cite{dream_derivations} is the principalization theorem on schemes with enough derivations, including the case where $X=\Spec(k\llbracket t_1\.t_n\rrbracket)$ is the spectrum of a regular complete local ring. In this section we will deduce from it the most general case -- the case of excellent schemes.

\subsection{Integral centers}\label{Mordsec}
As in \cite{dream_derivations} let $\cQ$ denote the set of increasing finite sequences of positive rational numbers and let $\cQ_1$ be the subset whose elements $\ud=(d_1\.d_n)$ are such that the numbers $e_1\.e_n$ defined by the rule $e_i=d_i\prod_{j=1}^{i-1}(e_j-1)!$ are integral. The multiorder of a weighted center $\cJ$ lies in $\cQ$. However, if $\cJ$ is the canonical center of an ideal, then $\ud\in\cQ_1$ due to the fact that it is obtained from the sequence $\ue$ of orders of iterated maximal contacts by a renormalization procedure. The resolution algorithm stops because the set $\cQ_1$ is well-ordered. The integrality condition defining $\cQ_1$ originates in Hironaka's work and is natural for the approach of iterated maximal contacts. However, it is not sharp and, moreover, there exists a simple combinatorial description of the precise set $\Mord$ of multiorders of canonical centers.

\subsubsection{The set $\Mord$}\label{Mord}
Let $\Mord$ denote the set of all increasing tuples of positive rational numbers $\ud=(d_1\.d_n)$ such that for each $1\le i\le n$ there exists natural numbers $a_1\.a_i$ such that $\frac{a_1}{d_1}+\dots+\frac{a_i}{d_i}=1$ and $a_i\neq 0$. We say that a center $\cJ$ is {\em integral} if $\mord(\cJ)\in\Mord$.  In particular, we will prove that $\Mord$ is precisely the set of invariants of canonical centers.

\begin{rem}\label{remord}
We stress that the integrality condition is weaker than requiring that the $\QQ$-ideal $\cJ$ is an ideal (that is, $\ud\in\NN^n$). For instance, consider the centers $\cJ=[x^5,y^7]$ and $\cJ'=[x^5,y^{15/2}]$ (also discussed in \cite[\S7.1]{ATW-weighted} and in Example~\ref{mordexam}(ii) below). Both are integral, but $\cJ$ is an ideal, while $\cJ'$ is not: $\cJ$ is generated by its rounding $(\cJ)=(x^5,x^4y^2,x^3y^3,x^2y^5,xy^6,y^7)=[x^5,y^7]$, while $\cJ'$ is strictly larger than the $\QQ$-ideal generated by $(\cJ')=(x^5,x^4y^2,x^3y^3,x^2y^5,xy^6,y^8)=[x^5,xy^6,y^8]$.

	
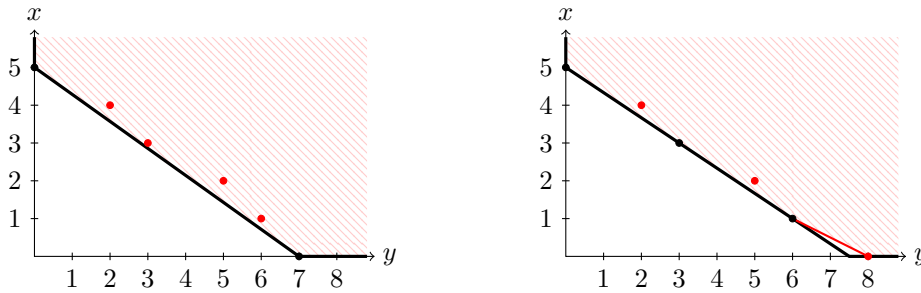
\begin{tikzpicture}[scale=0.5]
\draw[<->] (0,6)--(0,0)--(9,0);
\path[pattern=north west lines, pattern color=red!20!white,dashed]
	(0,5.8)--(0,5)--(7,0) -- (8.8,0)-- (8.8,5.8) -- (0,5.8);
\node at (0,6.4) {$ x $};
	\node at (9.4,0) {$ y $};
	
	\foreach \x in {1,...,8}
	\draw (\x,0.1) -- (\x,-0.1) node [below] {\x};
	
	\foreach \x in {1,...,5}
	\draw (0.1,\x) -- (-0.1,\x) node [left] {\x};
	
	\draw[very thick] (0,5.8)--(0,5)--(7,0) -- (8.8,0);

	\filldraw[black] (0,5) circle (2.5pt);
    \filldraw[red] (2,4) circle (2.5pt);
	\filldraw[red] (3,3) circle (2.5pt);
    \filldraw[red] (5,2) circle (2.5pt);
    \filldraw[red] (6,1) circle (2.5pt);
	\filldraw[black] (7,0) circle (2.5pt);
\hspace{200pt}

	\draw[<->] (0,6)--(0,0)--(9,0);
	
	\path[pattern=north west lines, pattern color=red!20!white,dashed]
	(0,5.8)--(0,5)--(3,3)--(6,1)--(8,0) -- (8.8,0)-- (8.8,5.8) -- (0,5.8);
	
	\node at (0,6.4) {$ x $};
	\node at (9.4,0) {$ y $};
	
	\foreach \x in {1,...,8}
	\draw (\x,0.1) -- (\x,-0.1) node [below] {\x};
	
	\foreach \x in {1,...,5}
	\draw (0.1,\x) -- (-0.1,\x) node [left] {\x};
	
	\draw[very thick] (0,5.8)--(0,5)--(6,1)--(7.5,0) -- (8.8,0);
    \draw[thick, red] (6,1)--(8,0);

	\filldraw[black] (0,5) circle (2.5pt);
    \filldraw[red] (2,4) circle (2.5pt);
	\filldraw[black] (3,3) circle (2.5pt);
	\filldraw[red] (5,2) circle (2.5pt);
    \filldraw[black] (6,1) circle (2.5pt);
	\filldraw[red] (8,0) circle (2.5pt);
	\end{tikzpicture}

The black lines are given by the conditions $\nu_{\cJ}=1$ and $\nu_{\cJ'}=1$, respectively, the red points satisfy the conditions $\nu_{\cJ}>1$ and $\nu_{\cJ'}>1$, respectively, and the red areas are the convex hulls of the sets of the lattice points on or above the black lines. The areas above the black lines correspond to the centers, while the red areas correspond to the centers generated by their roundings.
\end{rem}

\begin{rem}
Geometrically the meaning of the condition $\ud\in\Mord$ is as follows: consider the $n-1$-dimensional simplex $\sigma_n$ cut out from $\RR^n_{\ge 0}$ by the condition $\uw\cdot\ux=1$, and let $\sigma_i$ be the face given by the vanishing of $x_{i+1}\.x_n$. Then each set $\sigma_{i+1}\setminus\sigma_i$ contains an integral point.

\begin{tikzpicture}[scale=0.6]
	
	\draw[->] (0,0,0)--(7.5,0,0);
	\node at (7.9,0,0) {$ z $};
	
	\draw[<->] (0,4.5,0)-- (0,0,0)--(0,0,4.5);
	\node at (0,4.9,0) {$ y $};
	\node at (0,0,5.1) {$ x $};

	\draw[very thick] (6,0,0) -- (0,4,0) --  (0,0,3) -- (6,0,0);

    \path[pattern=north west lines, pattern color=black!20!white, dashed]
	(6,0,0) -- (0,4,0) --  (0,0,3) -- (6,0,0);

    \filldraw[black] (0,0,3) circle (2.5pt);
    \filldraw[black] (0,3,0.75) circle (2.5pt);
    \filldraw[black] (2.5,1.5,0) circle (2.5pt);

\end{tikzpicture}
\end{rem}

\begin{exam}\label{mordexam0}
The tuple $(4,\frac{16}{3},\frac{32}{5})$ lies in $\Mord$ because the equation $\frac{1}{4}x+\frac{3}{16}y+\frac{5}{32}z=1$ is satisfied at the points $(4,0,0)$, $(1,4,0)$ and $(2,1,2)$. In fact, in this case there is one more solution $(0,2,4)$. Thus, the center $\cJ=[x^4,y^{16/3},z^{32/5}]$ is integral. Using the consecutive maximal contacts $x,y,z$ one can easily check that $\cJ$ is the canonical center of $\cI=(x^4,xy^4,x^2yz^2)$.
\end{exam}

\subsubsection{The ideals $I_\ud$ and combinatorics of $\Mord$}
The following lemma is an elaboration of the observation in the above remark, and it provides the simplest motivation for introducing the set $\Mord$. From now on given a tuple $\ud\in\NN^n$ by $I_\ud\subset\NN^n$ we denote the ideal (in the sense of monoids) formed by all elements $\ua$ such that $\ua\cdot\ud^{-1}=\sum_{i=1}^n\frac{a_i}{d_i}\ge 1$.

\begin{lem}\label{comblem}
Let $\ud=(d_1\.d_n)\in\cQ$. Then $\ud\notin\Mord$ if and only if there exists $(d'_1\.d'_n)=\ud'>\ud$ such that $I_\ud\subseteq I_{\ud'}$.
\end{lem}
\begin{proof}
Assume first that $\ud\in\Mord$ and $\ud'>\ud$. Choose the minimal $i$ such that $d'_i>d_i$ and let $\ua=(a_1\.a_i,0\.0)\in\NN^n$ be such that $\ua\cdot\ud^{-1}=1$ and $a_i>0$. Then $\ua\cdot\ud'^{-1}<1$ and hence $\ua\in I_\ud\setminus I_{\ud'}$.

Conversely, assume that $\ud\notin\Mord$, and let us find $\ud'>\ud$ such that $I_\ud\subseteq I_{\ud'}$. Choose $i$ such that there exists no $\ua=(a_1\.a_i,0\.0)\in\NN^n$ with $a_i>0$ such that $\ua\cdot\ud^{-1}=1$. Switching $i$ with the maximal $j$ such that $d_i=d_j$ we can assume that either $i=n$ or $d_{i+1}>d_i$. For a rational $\veps\ge 0$ we define $\ud'=\ud_\veps$ by $d'_j=d_j$ if $j<i$ and $d'_j=d_i+\veps$ if $j\ge i$. Thus $\ud_\veps=(d_1\.d_{i-1},d_i+\veps\.d_i+\veps)$ and $\ud_\veps>\ud$ (lexicographically) whenever $\veps>0$. We will show that one can take $\ud'=\ud_\veps$ for a small enough positive $\veps$.

The main idea is now illustrated in the figures below on the case when $i=1$, $\ud=(\frac{14}{5},\frac{7}{2})\notin\Mord$, $\ud_\veps=(\frac{14}{5}+\veps,\frac{14}{5}+\veps)$ and $\ud_0=(\frac{14}{5},\frac{14}{5})$ correspond to the black, the blue and the dotted blue lines, and the red area on both figures is the convex hull of $I_\ud$.

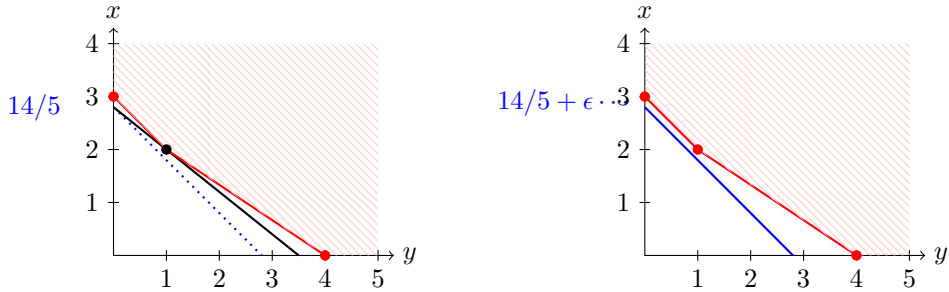
\begin{tikzpicture}[scale=0.7]

\draw[<->] (0,4.3)--(0,0)--(5.3,0);

\node at (0,4.6) {$ x $};
	\node at (5.6,0) {$ y $};
	\node[blue] at (-1.5,2.8) {$ 14/5$};
	\foreach \x in {1,...,5}
	\draw (\x,0.1) -- (\x,-0.1) node [below] {\x};
	
	\foreach \x in {1,...,4}
	\draw (0.1,\x) -- (-0.1,\x) node [left] {\x};
		
\draw[thick,blue, dotted] (0,2.8)-- (2.8,0);
	\draw[thick] (0,2.8) -- (3.5,0);
    \draw[thick, red] (0,3) -- (1,2) -- (4,0);
\path[pattern=north west lines, pattern color=red!20!white,dashed]
	(0,4)--(0,3)--(1,2)--(4,0)--(5,0) -- (5,4)-- (0,4);

	\filldraw[black] (1,2) circle (2.5pt);
\filldraw[red] (0,3) circle (2.5pt);
\filldraw[red] (4,0) circle (2.5pt);

\hspace{200pt}

\draw[<->] (0,4.3)--(0,0)--(5.3,0);

\node at (0,4.6) {$ x $};
	\node at (5.6,0) {$ y $};
	
	\foreach \x in {1,...,5}
	\draw (\x,0.1) -- (\x,-0.1) node [below] {\x};
	
	\foreach \x in {1,...,4}
	\draw (0.1,\x) -- (-0.1,\x) node [left] {\x};
	\node[blue] at (-1.5,2.9) {$ 14/5+\epsilon \cdots$};
	
%	\draw[thick] (0,2.8) -- (3.5,0);
    \draw[thick, red] (0,3) -- (1,2) -- (4,0);
\path[pattern=north west lines, pattern color=red!20!white,dashed]
	(0,4)--(0,3)--(1,2)--(4,0)--(5,0) -- (5,4)-- (0,4);

	\filldraw[red] (1,2) circle (2.5pt);
\filldraw[red] (0,3) circle (2.5pt);
\filldraw[red] (4,0) circle (2.5pt);

	\draw[thick,blue] (0,2.8)-- (2.8,0);
%	\draw[thick,blue] (0,2.9)-- (2.9,0);

	\end{tikzpicture}

First, we claim that for any single $\ua\in I_\ud$ one also has that $\ua\cdot\ud_\veps^{-1}\ge 1$ for a small $\veps>0$. Indeed, $$\ua\cdot\ud_\veps^{-1}-\ua\cdot\ud^{-1}=a_i\left(\frac{1}{d_i+\veps}-\frac{1}{d_i}\right)+\sum_{j=i+1}^na_j\left(\frac{1}{d_i+\veps}- \frac{1}{d_j}\right),$$ the first summand on the right tends to 0 as $\veps$ tends to 0, and the other summands tend to a non-negative number, which is positive whenever not all numbers $a_{i+1}\.a_n$ vanish. By our assumptions, it cannot happen that both $\ua\cdot\ud^{-1}=1$, $a_i\neq 0$ and $a_j=0$ for $i<j\le n$, and the claim follows. Second, if $\veps<C$, then $\ua\cdot\ud_\veps^{-1}\ge 1$ for any $\ua$ with $\min_i(a_i)\ge d_n+C$. Thus, it suffices to care only for finitely many points $\ua\in I_\ud$, and then the first claim implies that indeed $I_\ud\subseteq I_{\ud_\veps}$ for a small enough $\veps>0$.
\end{proof}

If $\cJ=[\ut^\ud]$ is a presentation of a center, then it follows from \cite[Lemma~2.3.5]{dream_derivations} that $(\cJ)=(\ut^{I_\ud})$ is generated by monomials $\ut^\ua$ with $\ua\in I_\ud$. In particular, we immediately obtain the following result.

\begin{cor}\label{roundlem1}
Let $X$ be a regular scheme and $\cJ=[\ut^\ud]$ a center on $X$. Then $\ud\notin\Mord$ if and only if there exists a tuple $\ud'>\ud$ of the same length such that $(\ut^{d})\subseteq(\ut^{\ud'})$.
\end{cor}

\subsubsection{The precise set of multiorders}\label{multisec}
Now let us prove that a center $\cJ$ is canonical for some ideal if and only if it is integral. Moreover, the question actually reduces to studying the relations between $\cJ$ and its rounding because if $\cJ$ is $\cI$-canonical (or just maximal), then the same is true for any intermediate ideal between $\cI$ and $\cJ$, and hence $\cJ$ is the canonical center of its rounding $(\cJ)$. In fact, using functoriality of the canonical centers with respect to completions and torus actions we will deduce this claim from its combinatorial (or toric) analogue proved in Lemma~\ref{comblem}.

\begin{theor}\label{maxcenterth}
The following conditions on a weighted center $\cJ$ on a regular scheme $X$ are equivalent:

(i) There exists an ideal such that $\cJ$ is its canonical center,

(ii) $\cJ$ is the canonical center of its rounding $\cI=(\cJ)$,

(iii) $\cJ$ is integral.

In particular, $\Mord$ is well-ordered and it is the precise subset of $\cQ_1$ in which the multiorders of ideals take values.
\end{theor}
\begin{proof}
We already observed that (i) and (ii) are equivalent in the obvious way, so let us prove that (ii)$\Longleftrightarrow$(iii). Assume first that $\ud=\mord(\cJ)\notin\Mord$. It suffices to prove that $\cJ$ is not the canonical center of $\cI$ locally, so we can assume that there exists a global presentation $\cJ=[\ut^\ud]$. By Lemma~\ref{comblem} there exists $\ud'>\ud$ with $\cJ'=[\ut^{\ud'}]$ such that $\cI\subseteq(\cJ')\subseteq\cJ'$, and hence $\mord(\cI)\ge\ud'>\ud$, as claimed.

Conversely, assume that $\ud\in\Mord$, but $\cJ$ is not the canonical center of $\cI=(\cJ)$, and let us obtain a contradiction. By definition, shrinking $X$ we can assume that $\cJ$ is not the maximal center of $\cI$. Thus, there exists an $\cI$-admissible center $\cJ'\neq\cJ$ such that $\mord(\cJ')\nless\ud$. Choose any $x\in V(\cJ)$ such that $\cJ'_x\neq\cJ_x$, then the same is true after pulling the situation back to $\hatX_x=\Spec(\hatcO_x)$, so we can assume that $X=\Spec(k\llbracket t_1\.t_m\rrbracket)$ and $\cJ=[\ut^\ud]$ is not the canonical center of its rounding $\cI$, where $\ut=(t_1\.t_n)$ is a partial family of parameters. At this stage one could use the derivations on $X$ and compute that $\ut$ is a maximal contacts flag and $\ud$ is the multiorder, but we will avoid the computations as follows. First, by the functoriality of roundings and canonical centers, the same claim holds already for $\Spec(k[t_1\.t_m])$ and the center $\cJ=[\ut^\ud]$ on it, so we can even assume that $X=\Spec(k[t_1\.t_m])$. Second, the center is equivariant for the action of the torus $T=\GG_m^m$, hence the same is true for its rounding $\cI=(\ut^{I_\ud})$ and for the canonical ideal $\cJ'=\cJ(\cI)$ of the rounding. Therefore, $\cJ'=[\ut^{\ud'}]$ for $\ud'>\ud$ and we obtain that $I_\ud\subseteq I_{\ud'}$, contradicting Lemma~\ref{comblem}.
\end{proof}

\begin{rem}
(i) Of course, the fact that the set $\Mord$ is contained in $\cQ_1$ and hence it is well ordered can be checked directly, though it is not completely trivial.

(ii) The actual bounds on the denominators in $\Mord$ are much better than in $\cQ_1$. For example, $d_1$ is natural and the denominator of $d_2$ is bounded by $d_1-1$ rather than just by a divisor of $(d_1-1)!$ Furthermore, this is also the precise set of the multiorder component of the invariant in the classical resolution. It is an interesting question if the theoretical bounds on the complexity of the classical resolution can be improved by taking the structure of $\Mord$ into account.
\end{rem}

\subsubsection{The set $\Mordone$}
Any $\ud=(d_1\.d_n)\in\Mord$ starts with a natural $d_1$, and by $\Mordone$ we denote the subset of $\Mord$ whose tuples start with $d_1\ge 2$. If $d_1=1$, then it does not affect all other integrality conditions in the definition of $\Mord$ and hence $(d_2\.d_n)\in\Mord$. Therefore, we obtain a splitting $\ud=(1\.1,d_m\.d_n)$ with $(d_m\.d_n)\in\Mordone$ and we will use the notation $(d_m\.d_n)=\ud_{>1}$.

\subsection{Quasi-excellent schemes}
A short survey about quasi-excellent schemes can be found in \cite[\S2.3]{Temkin-survey}. We recall here very briefly what is needed in this paper and include one more important example due to Nagata.

\subsubsection{The conditions of Grothendieck and Nagata}\label{GNsec}
Recall that a noetherian scheme $X$ is {\em quasi-excellent} or {\em qe} if it satisfies the conditions (N) and (G) below. If, in addition, $X$ is universally catenary, \cite[Tag 02J7]{stacks}, then it is called {\em excellent}.

\begin{itemize}
\item[(G)] Grothendieck's condition (which is essentially due to Hironaka): $X$ is a G-scheme if for any $x\in X$ the completion homomorphism $\cO_x\to\hatcO_x$ is a regular homomorphism.
\item[(N)] Nagata's condition: $X$ is an N-scheme if for any $Y$ of finite type over $X$ the regular locus of $Y$ is open.
\end{itemize}

Often (but not always), one can deal with the conditions separately, and this will (fortunately) be our case. The G-conditions will be used to prove things locally via algebraization (or descent) from the formal case, and the N-condition will be used to spread out results to neighborhoods of points. Now, let us consider the simplest examples of serious pathologies that can happen without the quasi-excellence assumption.

\subsubsection{Failure of the N-condition}\label{Nsec}
Nagata's condition can be severely violated already by curves. A comprehensive example by Gabber is worked out in detail in \cite{Laszlo}. Its idea is to take a countable base field $k$ with $\cha(k)\neq 2$ and infinite $k^\times/(k^\times)^2$, enumerate points $z_1,z_2,\dots$ on $C_0=\Spec(k[x])$ and construct (even explicitly) a family of finite covers $\dots\to C_2\to C_1\to C_0$ such that each $C_i$ is irreducible and nodal at all preimages of $z_i$, and for each closed point $z\in C_i$ in the preimage of $\{z_1\.z_i\}$ the morphism $f_i\:C_{i+1}\to C$ is \'etale with a single preimage over $z$ (i.e. $k(z)$ extends quadratically under $f_i$). It is then easy to see that the limit curve $C=\lim_i C_i$ is a notherian one-dimensional scheme with countably many closed points and all these points are nodes. Moreover, with a bit of care one can construct a closed embedding of the tower $C_i$ into a tower of finite covers $\dots S_2\to S_1\to S_0$ so that $S_i$ are surfaces essentially smooth over $k$ and $S_\infty=\lim_i S_i$ is a regular surface (in particular, noetherian). In fact, one can take $S_0$ to be the localization of $\Spec(k[x,y])$ along $C_0=V(y)$ and the map $g_i\:S_{i+1}\to S_i$ to be suitably ramified over the preimages of $z_{i+1}$ in $S_i$, \'etale over any other closed point, and with a single preimage over any point $z\in S_i$ in the preimage of $\{z_1\.z_i\}$. Of course, $S_\infty$ and its closed subscheme $C_\infty$ are not suited for resolution of singularities. For example, the normalization of $C_\infty$ is non-finite.

Now, let us show how this leads to the non-existence of canonical (or just maximal) centers. Let $\cI$ be the ideal defining $C_\infty$ in $S_\infty$, let $\eta$ be the generic point of $C_\infty$, let $C_\infty\setminus\{\eta\}=\{y_1,y_2,\.\}$ be the set of all closed points of $C_\infty$ and let $\cJ_i=m^2_{y_i}$. Then $\ord_\eta(\cI)=1$ and $\ord_{y_i}(\cI)=2$ for any $i$. In particular, $\ord_\cI$ is not upper semicontinuous. In addition, the localization at each $y_i$ is a local nodal curve and $\cJ_i$ is the canonical center of $\cI_{y_i}$. It follows that $\cJ'_n=\prod_{i=1}^n\cJ_i$ is a sequence of $\cI$-admissible centers with strictly increasing invariants -- the multiorder is constant and equals to $(2,2)$, while the support $V(\cJ'_n)=\{y_1\.y_n\}$ increases with $n$. There are no non-zero ideals contained in each $\cJ'_n$, hence there exists no maximal $\cI$-admissible center. In a sense, the maximal center should be the limit of $\cJ'_n$ and its support should be $\{y_1,y_2,\dots\}$.

\subsubsection{Failure of the G-condition}\label{Gsec}
A one-dimensional regular scheme which violates the G-condition, or just a non-excellent DVR, exists only in characteristic $p$. Here is a typical example: choose a field $k$ of characteristic $p$, let $y=f(t)\in k\llbracket t\rrbracket$ be transcendental over $k(t)$, and set $R=k(t,y^p)\cap k\llbracket t\rrbracket$. Then $R$ is a DVR whose completion $\hatR=k\llbracket t\rrbracket$ is generically inseparable over it. In characteristic zero there exist reduced one-dimensional local schemes whose completion is not reduced, but they are not normal. The smallest regular example of a non-G scheme in characteristic zero is constructed in dimension two using the same idea. Since it will serve us as an important test case later, let us briefly recall it here and refer to \cite[Appendix, Example~7]{Nagata} for details.

\begin{exam}\label{nonqeexam}
Choose a field $k$ of characteristic $0$ (or any $p\neq 2$) and an element $f(t)=\sum_{i=1}^\infty a_it^i\in k\llbracket t\rrbracket$ transcendental over $k(t)$. Set $y=(x-f(t))^2$ and $R=k(t,x,y)\cap k\llbracket t,x\rrbracket$. Nagata proved that $R$ is a regular local ring, $\hatR=k\llbracket t,x\rrbracket$ and the ideal $(y)$ is prime in $R$. Of course, $y$ is a square of a prime in $\hatR$, and hence the completion morphism $\Spec(\hatR)\to\Spec(R)$ has a non-reduced fiber. In addition, $R/(y)$ is a (non-normal) excellent local ring whose completion $\hatR/(y)$ is non-reduced.

We will also need a more concrete description of $R$. It is easy to see that $R=k[t,x,y,y_1,y_2,\dots]$, where $y_n=(y-(x-f_n)^2)/t^n$ are the rescaled approximations obtained from the truncations $f_n=\sum_{i=1}^{n-1} a_it^i$. Using this observation Nagata described the primes of $R$ by a straightforward computation, showed that the primes of height one are principal and $(x,t)$ is the only prime of height two, and deduced that $R$ is a regular local ring (in particular, noetherian).
\end{exam}

Finally, let us explain why the canonical center does not exist in this case.

\begin{exam}\label{nocanonical}
Consider the regular local surface $X=\Spec(R)$ and set $\cI=(y)$. Note that $\hatcI=(x-f)^2$ is a center on $\hatX=\Spec(\hatR)$ of invariant $(2)$, in particular, it is the canonical center of itself. On the other hand, $y$ is not a square in $R$, hence $\cI$ is not a center on $X$. The order of $\cI$ is 2, hence the maximal center, if exists, must be of multiorder $(2,d)$ with $d\ge 2$. Finally, there is an infinite sequence of $\cI$-admissible centers $\cJ_n=[(x-f_n)^2,t^{n}]$ of multiorder $(2,n)$, hence in this case the multiorder does not attain its maximum on the set of $\cI$-admissible centers. Of course, the reason why this happens is that $x-f$ is the only maximal contact to $\hatcI$, but it is not defined on $X$ and one can only take better and better approximations $x-f_n$ to $x-f$ leading to better and better approximations to the canonical center.
\end{exam}

\subsubsection{Failure of catenarity for qe schemes}\label{catenarysec}
A typical example of a qe non-catenary scheme is obtained as follows: take a field $k$ which admits an isomorphism $k\toisom k(t)$ (e.g. $k=\QQ(t_1,t_2,\dots)$), find an integral $k$-variety $X$ with a $k$-point $x$ and a non-closed point $y$ of codimension at least two such that $x$ is not in the closure of $y$ and there exists an isomorphism $\phi\:k(x)\toisom k(y)$ (e.g. $y$ is the generic point of a linear subspace of $\bbA^n$), localize $X$ at $\{x,y\}$, and glue $x$ and $y$ via $\phi$. Since the maximal specializing chains between the generic point $\eta$ and the points $x$ and $y$ are of different lengths, the obtained scheme $Z$ is non-catenary (and if $y$ is of codimension 1, the corresponding gluing $Z$ is catenary but not universally catenary). Note that $Z$ is not defined over $k$ or even any subfield $k_0\subseteq k$ such that $k/k_0$ is of finite transcendence degree.

In fact, this the only type of a failure that can happen in the qe case as we will now explain. By a theorem of Ratliff, see \cite{Ratliff} or \cite[tag~0AW1]{stacks}, a scheme is universally catenary if and only if it is formally catenary in the sense that the formal completion of any of its irreducible component at a point is equidimensional. In particular, this means that the existence of maximal prime ideal chains of different lengths between $p$ and $q$ indicates that the completion at $q$ contains irreducible components of different dimensions. In the case of G-schemes (non)-normality is preserved by completions, so any non-excellent qe scheme is non-normal, and, in fact, its normalization has branches of different dimensions over some point of $X$.

\begin{rem}\label{catenaryrem}
(i) One often uses the generality of qe schemes because this is the most general class which can possess a consistent theory of resolution of singularities, see \cite[${\rm IV}_2$, \S7.9.5]{ega}. This theorem of Grothendieck was the reason to introduce the class of qe schemes. However, up to normalization their resolution reduces to the case of excellent schemes.

(ii) In our main theorems \ref{principalizationth}, \ref{embeddedth} and \ref{resolutionth} we work with regular qe schemes when studying principalization and with formally equidimensional schemes when studying embedded resolution. In both cases, the scheme is automatically excellent, so we will use the word ``excellent'' instead of qe in formulations. Resolution of non-excellent qe schemes follows by using normalization, see Corollary~\ref{resolvecor}.

\end{rem}

\subsection{Algebraization}\label{algsec}
This section is devoted to descending canonical centers from formal completions to local excellent schemes. As one might expect, this involves a descent result with respect to regular morphisms, see Lemma~\ref{functorcenter}.

\subsubsection{Roundings}
We start with functoriality of roundings of $\QQ$-ideals, which essentially reduces to the fact that integral closures are compatible with normal morphisms, that is, flat morphisms with geometrically normal fibers.

\begin{lem}\label{functorrounding}
Any normal morphism of normal schemes $f\:X'\to X$ respects roundings of $\QQ$-ideals: if $\cJ$ is a $\QQ$-ideal on $X$ and $\cJ'=\cJ\cO_{X'}$, then $(\cJ')=(\cJ)\cO_{X'}$.
\end{lem}
\begin{proof}
The claim is local, so we can assume that $X$ is irreducible and there exists a finite cover $Y\to X$ with an irreducible normal source such that $\cJ_Y=\cJ\cO_Y$ is an actual integrally closed ideal. Indeed, locally $\cJ$ is generated by finitely many $\QQ$-ideals $[f_i^{1/n_i}]$ with $f_i\in\cO_X$ and passing to a finite cover we can assume that each $f_i^{1/n_i}$ also lies in $\cO_X$.

Once $\cJ_Y$ is an ideal, we have that $(\cJ)=\cJ_Y\cap\cO_X$ with the intersection taken in $\cO_Y$. Since the base change $f'\:Y'=Y\times_XX'\to X'$ is a normal morphism, we obtain that $Y'$ is normal and $\cJ_{Y'}=\cJ_Y\cO_{Y'}=\cJ_Y\otimes_{\cO_X}\cO_X'$ is integrally closed. Therefore, $(\cJ')=\cJ_{Y'}\cap\cO_{X'}$ in $\cO_{Y'}$ and since intersections of $\cO_{X'}$-submodules are compatible with pullbacks by the flatness of $f$, we obtain that $(\cJ')=(\cJ)\cO_{X'}$.
\end{proof}

Of course, the lemma fails for non-normal morphisms, including flat morphisms with non-reduced fibers. For example, it fails for $X=\Spec(k[t])$, $X'=\Spec(k[t^{1/2}])$ and $\cJ=[t^{1/2}]$.

\subsubsection{Descent of centers}
Using roundings we can now establish the following natural statement about functoriality of centers. The argument is a bit subtle and uses maximal contacts of centers, but I did not see a simper way to check this.

\begin{lem}\label{functorcenter}
Assume that $X'\to X$ is a regular morphism of regular schemes, $\cJ$ is a $\QQ$-ideal on $X$ and $\cJ'=\cJ\cO_{X'}$. If $V(\cJ)\subseteq f(X')$ and $\cJ'$ is a center, then $\cJ$ is a center and $\mord(\cJ)=\mord(\cJ')$.
\end{lem}
\begin{proof}
This can be checked locally at a point $x=f(x')\in V(\cJ)$, so assume that the schemes are local. Raising $\cJ$ to a power we can assume that it is generated by its rounding, in which case by our convention we identify $\cJ$ with the integrally closed ideal $(\cJ)$ and simply say that $\cJ$ is an ideal. By Lemma~\ref{functorrounding} $\cJ'$ is also an ideal, which is the ideal-theoretic pullback of $\cJ$.

Set $d=\ord_x(\cJ)$. Then $d=\ord_{x'}(\cJ')$ by \cite[Lemma~2.1.6(ii)]{dream_derivations}, and hence $\cJ'=[\ut^\ud]$ with $d=d_1$. We will use induction on the length $n$ of $\ud$ with the base $n=0$ being trivial. By Lemma~\ref{functorrounding}, $(\cJ^{1/d})\cO_{X'}=(\cJ'^{1/d})$, and since $\ord(\cJ'^{1/d})=1$ and the order is compatible with regular morphisms by \cite[Lemma~2.1.6(ii)]{dream_derivations}, we have that $\ord(\cJ^{1/d})=1$ and hence $(\cJ^{1/d})$ contains a parameter $t=t_1$.

Set $H=V_X(t)$ and $\cJ_H=\cJ|_H$. Since $t^d\in\cJ'$, the pullback $H'=V_{X'}(t)$ is a maximal contact to $\cJ'$ (see \cite[\S2.3.9]{dream_derivations}), and hence the pullback $\cJ'_{H'}=\cJ'|_{H'}$ of $\cJ_H$ is a center on $H'$ by \cite[Theorem~2.3.11(ii)]{dream_derivations}. Therefore, $\cJ_H$ is a center by the induction assumption. Choose $t_2\.t_n\in\cO_x$ such that their images in $\cO_H$ give rise to a presentation $\cJ_H=[t_2^{d'_2}\.t_n^{d'_n}]$ and hence also of its pullback $\cJ'_{H'}$. The latter implies that $d'_i=d_i$ for $2\le i\le n$ and $\cJ'=[t_1^{d_1}\.t_n^{d_n}]$ by \cite[Theorem~2.3.10(ii)]{dream_derivations}. Since $t_i\in\cO_X$ and $f$ is flat, this implies that also $\cJ=[\ut^\ud]$, and hence $\cJ$ is a center.
\end{proof}

\subsubsection{Formally canonical centers}\label{formcan}
Now let us discuss descent of canonicity. Given an ideal $\cI$ on a regular scheme $X$, we define the {\em formal multiorder} of $\cI$ at $x\in X$ by $\ford_x(\cI)=\mord(\hatcI_x)$. This makes sense as $\hatcO_x$ has enough derivations, hence $\hatcI_x$ possesses a canonical center by \cite[Theorem~3.3.14(i)]{dream_derivations}. One obtains a function $\ford_\cI:X\to\Mord$, though, like $\ord_\cI$ (which is its first component) it can have nasty properties in general, see \S\ref{Nsec}.

An $\cI$-admissible center $\cJ$ is called {\em formally canonical} if $V(\cJ)$ is the maximality locus of $\ford_\cI$ and for any $x\in V(\cJ)$ the $\hatcI_x$-admissible center $\hatcJ_x$ is the canonical one. We will later prove that for excellent schemes the notions of canonical and formally canonical centers (resp. multiorders and formal multiorders) are equivalent. At this stage we can only prove that the formal canonicity is stronger, so it will be used as an additional assumption while building the theory. In the same fashion, we will later establish functoriality of canonical centers for regular morphisms (which implies that canonical centers are formally canonical), but at this stage we only check functoriality of formal canonicity.

\begin{lem}\label{formalcenter}
Let $X$ be a regular scheme with an ideal $\cI$ and an $\cI$-admissible center $\cJ$, and let $f\:X'\to X$ be a regular morphism with pullbacks $\cI'=\cI\cO_{X'}$, $\cJ'=\cJ\cO_{X'}$.

(i) If $V(\cJ)\subset f(X')$ and $\cJ'$ is $\cI'$-canonical, then $\cJ$ is $\cI$-canonical.

(ii) If $\cJ$ is formally canonical, then it is canonical.

(iii) If $\cJ$ is the formally canonical center of $\cI$, then $\cJ'$ is the formally canonical center of $\cI'$, and the opposite implication holds whenever $V(\cJ)\subset f(X')$ .
\end{lem}
\begin{proof}
To establish (i) and (ii) assume that $\cJ$ is not canonical. Then the restriction $\cJ_U$ onto some open subscheme $U$ is non-trivial and non-maximal, that is, there exists an $\cI_U$-admissible center $\cJ_1$ on $U$ such that either $\mord(\cJ_1)>\mord(\cJ)$ or $\mord(\cJ_1)=\mord(\cJ)$ and $V(\cJ_1)\nsubseteq V(\cJ)$. Choose $x\in V(\cJ_1)$, and if the multiorders are equal ensure also that $x\notin V(\cJ)$. Looking at the formal completion at $x$ we see that $\mord(\hatcI_x)\ge\mord(\cJ)$ and if the equality holds, then $x\notin V(\cJ)$. In either case, $\cJ$ is not formally canonical at $x$, proving (ii). In the same manner, $\cJ'$ is not $\cI'$-canonical because its restriction to $U'=U\times_XX'$ is not maximal, as witnessed by the $\cI'_{U'}$-admissible center $\cJ_1\cO_{U'}$ at a point $x'$ over $x$.

(iii) Recall that for any $x'\in X'$ with $x=f(x')$ the completed homomorphisms $\hatcO_x\to\hatcO_{x'}$ is regular by \cite[Lemma~2.1.3]{dream_derivations}. Using \cite[Theorem~3.3.15(ii)]{dream_derivations} we obtain that $\ford_{\cI'}(x')=\mord(\hatcI'_{x'})=\mord(\hatcI_x)=\ford_\cI(x)$, and hence $\ford_{\cI'}=\ford_{\cI}\circ f$. Thus, if $\ford_\cI$ attains its maximum along $V(\cJ)$, then $\ford_{\cI'}$ attains its maximum along $V(\cJ')$, and the converse is true whenever $V(\cJ)\subset f(X')$. Choose any $x'\in V(\cJ')$ and set $x=f(x')$. It remains to note that by \cite[Theorem~3.3.15(ii)]{dream_derivations} and part (i) $\hatcJ_x$ is the canonical center of $\hatcI_x$ if and only if $\hatcJ_x\hatcO_{x'}=\hatcJ'_{x'}$ is the canonical center of $\hatcI_x\hatcO_{x'}=\hatcI'_{x'}$.
\end{proof}

\subsubsection{Formal descent}
Finally, we have all the needed tools to descend canonical centers from completions.

\begin{lem}\label{formaldescent}
Assume that $X$ is a regular local G-scheme with the closed point $x$ and an ideal $\cI$ such that the restriction $\cI_U$ of $\cI$ onto $U=X\setminus \{x\}$ possesses a formally canonical center. Then $\cI$ possesses a formally canonical center.
\end{lem}
\begin{proof}
Set $\hatX=\Spec(\hatcO_x)$, $\hatcI=\cI\cO_\hatX$, $\hatU:=\hatX\setminus\{x\}$ and $\hatcI_U=\hatcI|_\hatU$. Since $\hatX\to X$ is a surjective regular morphism and $\hatcI$ possesses a formally canonical center $\hatcJ$ by \cite[Theorem~3.3.15]{dream_derivations}, it suffices to prove that $\hatcJ$ is the formal completion of an $\cI$-admissible center $\cJ$ and use Lemma~\ref{formalcenter}(iii). Raising $\cI$ to a large enough integral power we can assume that $\hatcJ$ is an actual integrally closed ideal.

If $V(\hatcJ)=\{x\}$, then $\hatcJ=[\ut^\ud]$, where $\ud\in\NN^n$ and $t_1\.t_n$ is a sequence of regular parameters of $\hatcO_x$. Furthermore, $\hatcJ$ is open, hence for $\ut'$ close enough to $\ut$ we have that $\ut'^\ud\subset\hatcJ$ and hence $\hatcJ=[\ut'^\ud]$ by \cite[Corollary~2.3.12]{dream_derivations}. In particular, we can replace $\ut$ by a close enough tuple achieving that $\ut\subset\cO_x$. Then $\cJ=[\ut^\ud]$ is the required center on $X$.

Assume now that $\hatcJ$ is not open, and hence $\hatcJ_U=\hatcJ|_\hatU$ is non-empty. Then $\hatcJ_U$ is the canonical center of $\hatcI_U$ and by our assumption, $\hatcJ_U$ is the pullback of the canonical center $\cJ_U$ of $\cI_U$. Furthermore, $V(\hatcJ)$ is the schematic closure of $V(\hatcJ_U)$ because any center is normally flat along its reduction and hence does not have embedded components. So, we extend $\cJ_U$ to an ideal on $X$ in the same way -- $\cJ=\cJ_U\cap\cO_X$, which geometrically means that $V(\cJ)$ is the schematic closure of $V(\cJ_U)$. Then $\hatcJ$ is indeed the completion of $\cJ$ because schematic closures are compatible with flat morphisms, and it remains to recall that $\cJ$ is a center by Lemma~\ref{functorcenter}.
\end{proof}

\subsection{Extension on excellent schemes}\label{extcanonicalsec}
The remaining main problem is to extend local canonical centers to a neighborhood on an N-scheme. It would suffice to know that the function $\ford_\cI$ is upper semicontinuous, or just that $\ford_\cI$ is constant on a small enough non-empty open subset of the closure $V$ of $x$, but the problem is that derivations exist only formally locally and do not glue along $V$ in a reasonable way. For the same reason it is not clear how to algebraize formal maximal contacts. This forces us to introduce an analogue, in fact, an approximation of the maximal contacts flag which can be described in elementary terms not using derivations. In a sense, this is in the spirit of Hironaka's original approach via Tschirnhaus coordinates before derivations started to be heavily exploited in resolution. In addition, since Taylor expansions are not available too, we will have to work with certain graded pieces associated to the centers.

\subsubsection{Leading term of a center}
Let $\cJ$ be a center. By the {\em leading term} $\cT_\cJ=(\cJ)/(\cJ_{>1})$ we mean the quotient of the rounding of $\cJ$ by the ideal of all elements with $\nu_\cJ(f)>1$. In other words, it is the graded piece of weight one of the filtration induced by $\nu_\cJ$. Each presentation of $\cJ$ induces a canonical base of the leading term, which will be called the {\em monomial base} of the presentation:

\begin{lem}\label{conebase}
If $\cJ$ is a center and $V=V(\cJ)$ is its support with the reduced scheme structure, then $\cT_\cJ$ is a locally free $\cO_V$-module. Furthermore, if $\cJ$ possesses a global presentation $\cJ=[t_1^{d_1}\.t_n^{d_n}]$ and $w_i=d_i^{-1}$, then $\cT_\cJ$ is free with basis formed by monomials $\ut^\ua$ with $\ua\in\NN^n$ such that $\ua\cdot\uw=a_1w_1+\dots+a_nw_n=1$.
\end{lem}
\begin{proof}
The local description follows from \cite[Lemma~2.3.5]{dream_derivations} and implies the global one.
\end{proof}

\subsubsection{Leading term projection}
By $t_\cJ\:(\cJ)\to\cT_\cJ$ we denote the projection and call it the {\em leading term projection}. Note that $\cT_\cJ$ can be very small, even zero, because of the rounding, but it is large enough when $\cJ$ is integral. If $\cJ$ is $\cI$-admissible, then $t_\cJ(\cI)$ is {\em the leading term} of $\cI$ with respect to $\cJ$.

\begin{exam}\label{mordexam}
Let $X=\Spec(k[x,y])$.

(i) Note that $\cJ=[x^{n+1},y^{n+1}]$ is the canonical center of $\cI=(xy^n+y^{n+2})$ because the order is $n+1$ and both $x$ and $y$ are maximal contacts. The leading term $\cT_\cJ$ is spanned by $x^iy^j$ with $i+j=n+1$, and $t_\cJ(\cI)$ is the span of $xy^n$.

(ii) The ideals $\cI=(x^5+x^3y^3+y^7)$ and $\cI'=(x^5+x^3y^3+y^8)$ were studied in detail in \cite[\S7.1]{ATW-weighted}, and the canonical centers are $\cJ=[x^5,y^7]$ and $\cJ'=[x^5,y^{15/2}]$, respectively. We refer to Remark~\ref{remord} for a discussion about their roundings and figures. The leading term $\cT_\cJ$ is spanned by $x^5,y^7$, and the leading term of $\cI$ is maximal possible: $t_\cJ(\cI)=\cT_\cJ$. The leading term $\cT_{\cJ'}$ is spanned by $x^5,x^3y^3,xy^6$, and $t_{\cJ'}(\cI')$ is spanned by $x^5$ and $x^3y^3$.

(iii) If $\cJ=[x^4,y^{16/3},z^{32/5}]$ is the center from Example \ref{mordexam0}, then $\cT_\cJ$ is spanned by $x^4, xy^4, x^2yz^2,y^2z^4$. As we noticed, $\cJ$ is the canonical center of $\cI=(x^4, xy^4, x^2yz^2)$ and, of course, $t_\cJ(\cI)$ is spanned by these three monomials, which already suffice to determine the multiorder and guarantee that it is integral.
\end{exam}

\subsubsection{Tschirnhaus presentations}\label{Tsec}
Let $\cJ$ be an $\cI$-admissible center. We say that a presentation $\cJ=[t_1^{d_1}\.t_n^{d_n}]$ is {\em $\cI$-Tschirnhaus} if the following conditions hold: for any $1\le i\le n$ there exists an element $f_i\in\cI$ such that the decomposition $t_\cJ(f_i)=\sum_\ua b_\ua\ut^\ua$ accordingly to the monomial basis from Lemma~\ref{conebase}
\begin{itemize}
\item[(i1)] contains a non-zero term with $b_\ua=1$ and $\ua=(c_1\.c_i,0\.0)$, $c_i\neq 0$,
\item[(i2)] does not contain non-zero terms with $\ua$ starting with $(c_1\.c_{i-1},c_i-1)$.
\end{itemize}

Note that conditions (i1) immediately imply that $\nu_\cJ(\ut^\ua)=1$ and hence $\ud\in\Mord$ and $\cJ$ is integral. Condition (i2) is a sort of Tschirnhaus condition and it will be achieved in a construction below via the classical Tschirnhaus substitution. We will now prove that existence of a Tschirnhaus presentation is equivalent to canonicity of the center, and here is the first half of this result.

\begin{lem}\label{tschlem}
If $\cI$ is an ideal on a local regular G-scheme $X$, and $\cJ$ is a center which possesses an $\cI$-Tschirnhaus presentation $\cJ=[\ut^\ud]$, then $\cJ$ is formally $\cI$-canonical.
\end{lem}
\begin{proof}
We should prove that $\hatcJ=\cJ\hatcO_x$ is the canonical center of $\hatcI=\cI\hatcO_x$, where $x\in X$ is the closed point. Of course, $\ut$ also gives rise to an $\hatcI$-Tschirnhaus presentation of $\hatcJ$, so passing to the completions we can simplify the setting and assume that $X=\Spec(A)$, where $A$ is a complete local regular ring, say, $A=C\llbracket\ut\rrbracket$ with $C=k\llbracket\us\rrbracket$ and parameters $\us=(s_1\.s_m)$ that do not show up in the presentation of $\cJ$.

Now we will want to use results about maximal contacts from \cite{dream_derivations}, and we will first show that $t=t_1$ can be replaced by a maximal contact element. Set $d=d_1$ and choose $f\in\cI$ such that $t_\cJ(f)=t^d+a_2t^{d-2}+\dots+a_d$. Since the order of $f$ is $d$ one can find a presentation $f=ut^d+\sum_{i=1}^{d}a_it^{d-i}$, where $u\in 1+m_A$ is a principal unit and $a_i\in C\llbracket t_2\.t_n\rrbracket$. Indeed, one can just reduce $f$ by $t$ using the Gr\"obner reduction with respect to the graded lex order of monomials. Note that $u^{1/d}\in A$, so replacing $t$ by $u^{1/d}t$ we can assume that $u=1$. This change of the presentation does not affect the monomial base of the graded cone $\cT_\cJ$ , hence the new tuple $\ut$ gives rise to a Tschirnhaus presentation as well. The element $t'=\partial_t^{d-1}(f)/d!=t+a_1/d$ lies in $\cD_X^{(\le d-1)}(\cJ)\subseteq\cJ^{1/d}$, hence $a_1\in\cJ^{1/d}$ and $a_1t^{d-1}\in\cJ$. Moreover, by the Tschirnhaus assumption the leading term of $a_1t^{d-1}$ vanishes, and hence $a_1t^{d-1}\in\cJ_{>1}$. It follows that replacing $t$ by the maximal contact $t'$ does not affect the monomial base of $\cT_\cJ$, and hence the new presentation satisfies the same Tschirnhaus condition.

Once we have achieved that $t=t_1$ is a maximal contact, the rest follows by induction. Set $H=V(t)$ and $\cI_1=\cC_{\partial_t}(\cI)|_H$. The definition of Tschirnhaus presentations is easily seen to be compatible with the restriction to $V(t_1)$, that is, the presentation $\cJ_H^{(d-1)!}=[t_2^{d_2(d-1)!}\.t_n^{d_n(d-1)!}]$ is $\cI_1$-Tschirnhaus. By induction on $n$ we can assume that $\cJ_H^{(d-1)!}$ is the canonical center of $\cI_1$, and then \cite[Lemma~3.3.11 and Theorem~3.3.14(i)]{dream_derivations} implies that $\cJ$ is the canonical center of $\cI$. This finishes the proof, but we also note for completeness that similar changes of variables can be done inductively for $t_2, t_3\dots$ to transform the Tschirnhaus presentation to a maximal contacts flag without affecting the monomial basis of $\cT_\cJ$.
\end{proof}

\subsubsection{Construcitve canonicity}
Finally, we can use Tschirnhaus coordinates to obtain a constructive criterion of canonicity.

\begin{theor}\label{tschirnhausexist}
Assume that $X$ is a local regular G-scheme, $\cI$ is an ideal on $X$. Then the following conditions on an $\cI$-admissible center $\cJ$ are equivalent and imply that $\cJ$ is integral:

(i) $\cJ$ is formally canonical,

(ii) $\cJ$ is canonical,

(iii) $\cJ$ possesses an $\cI$-Tschirnhaus presentation.
\end{theor}
\begin{proof}
The implications (iii)$\Longrightarrow$(i)$\Longrightarrow$(ii) are covered by Lemmas~\ref{tschlem} and \ref{formalcenter}(ii), so it remains to prove that (ii) implies (iii). Assume that $\cJ$ is the canonical center of $\cI$ and set $\ud=\mord(\cJ)$. We start with a presentation $\cJ=[\ut^\ud]$ and we will gradually transform it to an $\cI$-Tschirnhaus one. By induction, we can assume that the two Tschirnhaus conditions are satisfied for $1\.i-1$, and we should prove that replacing $t_i\.t_n$ by $t'_i\.t'_n$ one can obtain another presentation which satisfies also (i1) and (i2). We start with the following claim whose proof uses the same combinatorics as in our study of integral centers in \S\ref{Mordsec}. In particular, we again use the notation $\ud_\veps=(d_1\.d_{i-1},d_i+\veps\.d_i+\veps)$ for a rational $\veps\ge 0$.

Claim: the leading term $t_\cJ(\cI)$ contains an element $t_\cJ(f)$ whose decomposition $\sum_\ua b_\ua\ut^\ua$ involves a non-zero term $b_\uc t^\uc$ such that $\uc\cdot\ud^{-1}_0\le 1$ and $c_j\neq 0$ for some $j\ge i$.

Assume, to the contrary, that the claim fails and no such element exists. Let us show that for each $f\in\cI$ and a rational small enough $\veps>0$ depending on it, one has that $f\in\cJ_\veps=[\ut^{\ud_\veps}]$. Indeed, this can be checked after formal completion, when there exists a decomposition $f=\sum_{\uc\in\NN^n}b_\uc\ut^\uc$ such that either $b_\uc=0$ or $\uc\cdot\ud\ge 1$, and if $\uc\cdot\ud=1$, which happens if and only if the term contributes to the leading term $t_\cJ(f)$, then $\uc\cdot\ud_0>1$ and hence also $\uc\cdot\ud_\veps>1$ for a small $\veps$. Now, existence of such an $\veps=\veps(f)$ is checked similarly to the proof of Lemma~\ref{comblem}: it is enough to test finitely many $\uc$ with $\uc\cdot\ud>1$, and this is done by the same type of estimate. Thus, for a small $\veps>0$ we have that $\cJ_\veps$ contains a finite set of generators of $\cI$, and hence $\cJ_\veps$ is an $\cI$-admissible center with $\mord(\cJ')=\ud_\veps>\ud=\mord(\cJ)$. The contradiction proves that an $f\in\cI$ as above exists, finishing the proof of the claim.

Now we will use $f$ to produce an appropriate coordinate change. First, renormalizing $f$ (i.e. multiplying it by a unit) we can also achieve that $b_\uc=1$. Second, since $\uc\cdot\ud^{-1}=1\ge\uc\cdot\ud_0^{-1}$, we necessarily have that $c_j=0$ for any $j$ such that $d_j>d_i$. In other words, $t^\uc=xy$, where $x=\prod_{j=1}^{i-1}t_j^{c_j}$ and the monomial $y=\prod_{j=i}^mt_j^{c_j}$ involves only coordinates of weight $1/d_i$, that is, $d_i=\dots=d_m$. Replacing $t_j$ by $t_j+b_it_i$ for $i+1\le j\le m$ and general enough $b_i\in\QQ$ we achieve that after another renormalization $t_\cJ(f)$ contains a non zero monomial of the form $xt_i^{e}$, where $e=\sum_{j=i}^mc_j$. Clearly, this change does not affect the Tschirnhaus conditions for $1\.i-1$ and the condition (i1) is now satisfied with $\uc'=(c_1\.c_{i-1},e,0\.0)$.

Next, let $xt_i^{e-1}y_1\.xt_i^{e-1}y_l$ with $y_j$ being monomials in $t_{i+1}\.t_n$ be all terms of $t_\cJ(f)$ starting with $xt_i^{e-1}$. To get rid of them, we perform the classical Tschirnhaus substitution $t'_i=t_i-e^{-1}\sum_{j=1}^l xy_j$ (and replace $t_i$ by $t'_i$). The new tuple gives rise to a presentation of $\cJ$ by \cite[Corollary~2.3.12]{dream_derivations} because $\nu_\cJ(t_i)=\nu_\cJ(y_j)=1-\nu_\cJ(xt_i^{e-1})$ and hence $t'^{d_i}_i\in\cJ$. This substitution does not affect the conditions for $1\.i-1$ and (i1), and guarantees that (i2) is satisfied too.
\end{proof}

\begin{rem}
The intuitive meaning of the theorem is that the canonical center $\cJ$ is naturally determined by the condition that $t_\cJ(\cI)$ spans a subspace of $\cT_\cJ$, which is not contained in the leading term of any proper subcenter, and the Tschirnhaus condition is a convinient technical way to guarantee this.
\end{rem}

\subsubsection{Extension of canonicity}
As an application of Tschirnhaus coordinates we obtain an extension result for general excellent schemes:

\begin{cor}\label{extendcor}
Assume that $X$ is an excellent regular scheme, $x\in X$ is a point, $\cI$ is an ideal on $X$ and $\cJ_x$ is the canonical center of $\cI_x$. Then there exists a neighborhood $U$ of $x$ such that $\cJ_x$ extends to the canonical center $\cJ_U$ of $\cI_U=\cI|_U$.
\end{cor}
\begin{proof}
By Theorem \ref{tschirnhausexist} there exists a Tschirnhaus presentation $\cJ_x=[\ut^\ud]$, where $t_i$ are global functions on a neighborhood $U$ of $x$. By the N-condition, shrinking $U$ we can assume that the scheme $V=V_U(t_1\.t_n)$ is regular and hence we obtain a center $\cJ_U=[\ut^\ud]$ on $U$. Shrinking $U$ again we also achieve that $\cJ_U$ is $\cI_U$-admissible. In view of Theorem~\ref{tschirnhausexist} it suffices to prove that after an additional shrinking of $U$ we can also achieve that the presentation is $\cI_U$-Tschirnhaus everywhere along $V$. Choose an element $f_i\in\cI_x$ such that the decomposition $t_\cJ(f_i)=\sum_\ua b_\ua\ut^\ua$ accordingly to the monomial basis satisfies the conditions (i1) and (i2) from \S\ref{Tsec}. Then the same decomposition holds along $V$ in the neighborhood $U$ on which $f_i$ is defined, and hence the conditions hold in $U$ as well.
\end{proof}

\subsubsection{A non-excellent model and some conclusions}\label{conclusions}
In order to prove the extension lemma we had to provide a very constructive and explicit description of canonical centers in terms of Tschirnhaus coordinates. The latter were constructed using a formal presentation and the theory of maximal contact, so the following natural question arises: can one construct them directly obtaining a simpler and more direct proof of existence of canonical centers, which will uniformly work on all excellent schemes? Here is an informal argument why this is impossible. The argument should somehow exclude local schemes whose completion morphism is not regular, and this cannot be encoded in simple invariants like order of ideal, etc. Let us illustrate this on a simple model case from Example~\ref{nocanonical} where no maximal center exists but partial approximations are $\cJ_n=[z_n^2,t^{n}]$, where $z_n=x-f_n$. Of course, $z_n,t$ are not Tschirnhaus coordinates, as detected by a monomial $cz_nt^n$ in the leading term of $\cJ_n$, and this forces us to make a correction of $z_n$, say, replace it by $z_{n+1}$. However, the leading term contains no more specific information, which would allow to directly generate any better approximation such as $z_{n+2}$, and on the level of Tschirnhaus coordinates this will only be detected on the next step, when we look at the image of $y$ in the leading term of $\cJ_{n+1}$. By a ``mysterious'' coincidence built in our example all these images in the leading terms will turn out to be a square, forcing the algorithm to produce the sequence of $\cI$-admissible centers with unbounded multiorder and run without stop.

In view of the above example, it seems plausible that a direct approach based on Tschirnhaus coordinates cannot succeed, and one has to use either derivations (as in the modern approaches via maximal contacts) or Weierstrass division/preparation on the formal level (as in the original Hironaka's approach). For example, the correct formal center $x-f$ could be produced either by deriving $y$ or by a Weierstrass division in formal series, which clears off all terms $c_izt^i$ at once, and not only the minimal one. In either case, one then has to use a non-trivial input to descend this to general excellent schemes without enough derivations.

\subsection{The main principalization results}

\subsubsection{The canonical stratification}
Recall that existence of the canonical stratification of $\cI$ implies that for any open subscheme $U\subseteq X$ the restriction $\cI_U$ possesses the canonical center, $\mord_\cI$ exists, is upper semicontinuous and takes values in $\Mord$, see \cite[Lemma~3.3.7]{dream_derivations} and Theorem~\ref{maxcenterth}.

\begin{theor}\label{canonicalcenter}
Let $X$ be a noetherian regular excellent stack of characteristic zero and let $\cI$ be an ideal on $X$, then

(i) $\cI$ possesses a canonical stratification $V(\cI)=\coprod_i V(\cJ_i)$. In particular, the multiorder function $\mord_\cI$ is well-defined, upper semicontinuous and takes values in $\Mord$.

(ii) If $f\:X'\to X$ is a regular morphism, then $\coprod_iV(\cJ_i\cO_{X'})$ is the canonical stratification of $\cI'=\cI\cO_{X'}$ and $\mord_\cI=\mord_{\cI'}\circ f$.

(iii) Let $\cJ=\cJ(\cI)$ be the canonical center, let $(d_1\.d_n)=\mord_X(\cI)$, let $N>0$ be a natural number such that $N/d_i\in\NN$ for $1\le i\le n$ and let $\cI'$ be the transform of $\cI$ with respect to $X'=\Bl_{\cJ^{1/N}}(X)\to X$. Then $\mord_{X'}(\cI')<\mord_X(\cI)$.
\end{theor}
\begin{proof}
In view of the functoriality claim (ii) it suffices to prove the theorem for schemes, so we assume in the sequel that $X$ an $X'$ are schemes. The main claim is that $\cI$ possesses canonical stratification. We will use notherian induction, so assume that $\cI_W=\cI|_W$ possesses a canonical stratification for an open subscheme $W\subsetneq X$, and let us prove that this is also true for a larger open subscheme. Choose any generic point $\eta$ of the complement $X\setminus W$ and consider the localization $X_\eta$. Then $X_\eta\setminus \eta$ lies in $W$ and hence possesses canonical stratification by the functoriality. Therefore, $\cI_\eta$ possesses a canonical center by Lemma~\ref{formaldescent}. By Corollary~\ref{extendcor} there exists a neighborhood $U$ of $\eta$ such that $\cI_U$ possesses a canonical center $\cJ_U$ with $\eta\in V(\cJ_U)$. Shrinking $U$ we can achieve that $U\setminus V(\cJ_U)\subset W$, and then $\cI_U$ possesses a canonical stratification. Therefore the union $U\cup W$ also possesses a canonical stratification, and $W\neq U\cup W$ because $\eta\in U\setminus W$.

Now, let us deduce the other claims. By Theorem~\ref{maxcenterth} $\mord_\cI$ takes values in $\Mord$ and it is upper semicontinuous by \cite[Lemma~3.3.7]{dream_derivations}, proving the other claims of (i). In (ii) it suffices to prove that if $\cJ$ is the canonical center of $\cI$ and $\cJ\cO_{X'}\neq\cO_{X'}$, then $\cJ\cO_{X'}$ is the canonical center of $\cI'$. Then the claim about the stratification follows by induction and implies that $\mord_\cI=\mord_{\cI'}\circ f$. It remains to recall that formal canonicity is equivalent to canonicity by Theorem~\ref{tschirnhausexist}, and functorility for formal canonicity has been already proved in Lemma~\ref{formalcenter}(iii). Finally, root blowings up of centers are compatible with regular morphisms, hence the drop of the multiorder in (iii) can be checked on formal completions of $X$. In this case we have enough derivations, so the result was established in \cite[Theorem~3.4.4]{dream_derivations}.
\end{proof}

\subsubsection{Dream principalization}\label{princsec}
Given a regular excellent stack $X$ with an ideal $\cI$ we now define the {\em dream principalization sequence} $X_n\to\dots X_0=X$ inductively by the rules $\cI_0=\cI$, $\cI_{i+1}=\cI_i\cO_{X_{i+1}}(\cJ_i\cO_{X_{i+1}})^{-1}$, $\cJ_i=\cJ(\cI_i)$ and $X_{i+1}=\Bl_{\cJ_i^{1/N_i}}(X_i)$, where $N_i$ is the smallest positive integer such that $N_i/d_i\in\NN$ for $1\le i\le n$ and $(d_1\.d_n)=\mord(\cJ_i)$. The regularizing root order $N_i$ is chosen minimal possible for the sake of canonicity. Thus the whole procedure is just to iteratively find the maximal admissible center, blow it up with an appropriate regularizing root construction and transform the ideal.

\begin{proof}[Proof of Theorem \ref{principalizationth}]
We deduce everything from Theorem~\ref{canonicalcenter}. Indeed, it asserts that the canonical centers and hence the principalization sequence exist, each blowing up decreases the multiorder of the transform of the ideal, and the sequence terminates after finitely many steps because the set of invariants $\Mord$ is well-ordered. In the end the invariant attains the minimal value $(0)$ and hence the ideal is trivial. So, we obtain (i). By induction on $i$ one easily sees that at each step one obtains the decomposition $\cI\cO_{X_i}=\cI_i\prod_{j=0}^{i-1}\cJ_{j}\cO_{X_i}$, where all ideals except $\cI_i$ are invertible. Since $\cI_n=(1)$ we obtain that $\cI\cO_{X_n}=\prod_{j=0}^{n-1}\cJ_{j}\cO_{X_n}$ is invertible, yielding (iii).

Since canonical stratifications are compatible with regular morphisms, by induction on $i$ we obtain that each center of the principalization sequence of $X$ with a non-empty pullback over $Y$ pulls back to a center of the principalization sequence of $Y$, yielding (iii). In particular, the functoriality implies that the sequence is trivial over the complement of $V(\cI)$, concluding the proof.
\end{proof}

\subsubsection{Dream embedded resolution}
As another application of Theorem~\ref{canonicalcenter} we obtain resolution of reduced subschemes of $X$. Assume now that $Z\into X$ is a closed reduced subscheme of constant codimension in $X$, then we define the {\em dream embedded resolution sequence} $X_n\to\dots X_0=X$ inductively by the rules $Z_0=Z$, $Z_{i+1}\neq\emptyset$ is the strict transform of $Z_i$ under $X_{i+1}\to X_i$, $\cI_i\subset\cO_{X_i}$ is the ideal defining $Z_i$, $\cJ_i=\cJ(\cI_i)$ and $X_{i+1}=\Bl_{\cJ_i^{1/N_i}}(X_i)$, where $N_i$ is the smallest positive integer such that $N_i/d_i\in\NN$ for $1\le i\le n$ and $(d_1\.d_n)=\mord(\cJ_i)$. The whole procedure is just to iteratively find the maximal $\cI_i$-admissible center $\cJ_i$, blow it up with an appropriate regularizing root construction and transform the subscheme. By definition, the process ends when $V(\cJ_i)$ contains a generic point of $Z_i$ and hence the next strict transform $Z_{i+1}$ would not be birational to $Z_i$.

\begin{theor}\label{embeddedth}
Let $X$ be a noetherian regular excellent stack of characteristic zero and $Z\into X$ a closed reduced subscheme of constant codimension, then:

(i) The dream embedded resolution sequence of $X$ is finite and ends with a regular $Z_n$.

(ii) The sequence $X_n\to\dots\to X_0=X$ is trivial over the regular locus of $Z$.

(iii) The induced sequence $Z_n\to\dots\to Z_0=Z$ is a resolution of $Z$ by a sequence of stack-theoretic modifications.

(iv) The sequence is functorial with respect to any regular morphism $f\:X'\to X$: the dream embedded resolution of $Z'=Z\times_XX'$ in $X'$ is obtained from the dream embedded resolution of $Z$ in $X$ by pulling it back to a sequence of normalized root blowings up of $X'$ and removing all blowings ups with empty centers..

(v) The induced resolution sequence $Z_n\to\dots\to Z_1\to Z_0=Z$ is independent of the embedding $Z\into X$.
\end{theor}
\begin{proof}
(i) Consider the $i$-th blowing up. The strict transform $Z_i$ of $Z_{i-1}$ is the schematic closure of $Z_{i-1}\setminus V(\cI_{i-1})$, hence it is a closed subscheme of the scheme defined by $\cI'_i=\cI_{i-1}\cO_{X_{i}}(\cJ_{i-1}\cO_{X_{i}})^{-1}$, which implies that $\cI'_i\subseteq\cI_i$ and $\mord_{X_i}(\cI_i)\le\mord_{X_i}(\cI'_i)<\mord_{X_{i-1}}(\cI_{i-1})$. Thus, the invariant $\mord_{X_i}(\cI_i)\in\Mord$ drops and hence the algorithm stops after finitely many steps. Once it ends with some $Z_n$, the canonical center contains a generic point $\eta$. Since $Z_n$ is reduced at $\eta$, it is regular at $\eta$ and hence $\mord_\eta(\cI_n)=(1\.1)$ is of length $c=\dim(\cO_{X,\eta})$. Since $Z_n$ is also reduced of codimension $c$ at any other generic point $\eta'$ we also have that $\mord_{\eta'}(\cI_n)=(1\.1)$, and hence $Z\subseteq V(\cJ_n)$, which implies that the equality holds and $Z$ is regular.

Claims (ii)--(iv) follow obviously, so it remains to check (v). Recall first the re-embedding principle from \cite{ATW-weighted}: if $X\into X'$ is a closed immersion of a constant codimension $c$ with a regular $X'$, then the induced resolutions of $Z$ coincide and the multiorders in $X'$ are obtained from the multiorders in $X$ by adding $c$ ones, namely, $\mord_{X'_i}(\cI'_i)=(1\.1,\mord_{X_i}(\cI_i))$. Indeed, locally $X=V_{X'}(s_1\.s_c)$ for a regular sequence of parameters in $X'$, hence $\cJ(\cI')=[s_1\.s_c,t'^{d_1}_1\.t'^{d_n}_n]$, where the images $t_i\in\cO_X$ give rise to a presentation $\cJ(\cI)=[\ut^\ud]$. Moreover, it is easy to see that this is preserved by blowings up of centers on $X$ (and the corresponding centers on $X'$): $X_i=V_{X'_i}(s_1\.s_c)$, hence the multiorders agree throughout the sequence.

Assume now that $Z\into X'$ is another embedding of constant codimension and let $Z'_{n'}\to\dots\to Z'_1\to Z$ be the induced resolution. Without limitation of generality, $Z$ is of larger codimension in $X$, and using the re-embedding principle, we can replace $X$ by an appropriate $\bbA^n_X$ achieving that $Z$ is of the same codimension in both $X$ and $X'$. It now suffices to check that the sequences are equal after localization at a point $z\in Z$, including equality of multiorders at each blow up step $Z_{i+1}\to Z_i$, which guarantees synchronization at different points. The claim about local sequences can be checked after formal completion, reducing to the case when $Z=\Spec(A)$ with a local complete $A$ and residue field $k$. Then $Z\into X$ and $Z\into X'$ factor through a minimal embedding $Z\into X_0$ with $X_0=\Spec(k\llbracket t_1\.t_n\rrbracket)$ (see \S\ref{minemb} below), and it remains to use the re-embedding principle once again.
\end{proof}

\begin{rem}
The algorithm does not provide strong resolution of non-reduced schemes in the sense of Hironaka, as the following example by Dan Abramovich shows. Let $X=\Spec(k[x,y,z])$ and $\cI=(x^2,y^2,xyz)$. The integral closure of $\cI$ equals $[x^2,y^2]$, hence $\mord_z(\cI)=(2,2)$ for any point $z\in Z=V(x^2,y^2,xyz)$, the principalization blows up the center $\cJ=[x^2,y^2]$, whose support is the whole reduction $\oZ=\Spec(k[z])$ of the non-reduced scheme $Z$, so the embedded resolution algorithm does not modify $Z$ at all. Of course, $Z$ is not normally flat along $\oZ$, and even has an embedded component at the origin.
\end{rem}

\section{Non-embedded resolution}\label{nonembsec}
In order to extend non-embedded resolution to excellent schemes which cannot be embedded into regular ones, we will have to extend the theory of centers to such schemes, including the non-normal (or even non-reduced) ones. We will use roundings of integral centers or the closed subschemes corresponding to them which will be called tubes. Such schemes can be defined independently of an embedding and the first two subsections are devoted to studying them.

\subsection{Tubes}\label{tubesec}
Recall that a center $\cJ$ is integral if $\mord(\cJ)\in\Mord$. An integral center $\cJ$ is controlled by the rounding $\cJ_0=(\cJ)$ well enough. Our next goal is to describe it completely in terms of the scheme $V(\cJ_0)$, which is an archetypical example of tubes we introduce below.

\subsubsection{Tubes and valuative filtrations}
In general, a tube can be thought of as a special (in a sense, nicest) nilpotent thickening of a regular scheme.

\begin{defin}
(i) A scheme $V$ is called a {\em split tube} of {\em width} $\ud=(d_1\.d_n)\in\Mord_{>1}$ if there exist functions $t_1\.t_n\in\cO_V$, called {\em nilpotent parameters} of $V$, such that the following conditions are satisfied:

\begin{itemize}
\item[(0)] The ideal $\cN=(t_1\.t_n)$ defines a regular closed subscheme $\oV=V(\cN)$.
\item[(1)] $\ut^{I_\ud}=0$.
\item[(2)] The images of monomials $\ut^\ua$ with $\ua\in\NN^n\setminus I_\ud$ in $\gr_\cN(\cO_V)=\oplus_{i=0}^\infty\cN^i/\cN^{i+1}$ are linearly independent over $\cO_{\oV}$.
\end{itemize}

(ii) A scheme $V$ is a called a {\em tube} if there exists $\ud\in\Mord_{>1}$, called the {\em width} of $V$, such that $V$ is locally a split tube of width $\ud$. We use the notation $\ud=\mord(V)$ because the width can also be interpreted as the tuple of orders of vanishing of the parameters.

(iii) The decreasing {\em valuative filtration} $\cF=\cF_V$ on $\cO_V$ induced by $\ut$ is defined as follows: for $r\in[0,1]$ the filtered piece $\cF_r$ is generated by all monomials $\ut^\ua$ with $\ua\cdot\ud^{-1}\le r$.
\end{defin}

\subsubsection{Basic properties of tubes}
We will later prove that both $\mord(V)$ and $\cF_V$ are independent of the choice of nilpotent parameters, but this is not so simple and will be deduced from analogous properties of weighted centers. We start our study of tubes with recording a few properties which are very simple.

\begin{rem}\label{tuberem}
For any tube $V$ its reduction $\oV$ is regular and $V$ is normally flat along $\oV$. In fact, it suffices to check this locally, when $V$ is split of width $\ud$ with nilpotent parameters $\ut$, and in such a case we can say more:

(i) It follows from condition (1) that $\cN=(\ut)$ is the nilradical of $\cO_V$ and $\oV$ is the reduction of $V$.

(ii) Condition (2) has a simple meaning and can be reformulated as follows. Since the associated graded algebra $\gr_\cN(\cO_V)$ is generated by monomials and $\ut^\ua$ vanish for $\ua\in I_\ud$, the monomials with $\ua\in\NN^n\setminus I_\ud$ give rise to an $\cO_\oV$-basis of $\gr_\cN(\cO_V)$. In particular, each graded piece is a free module on non-zero monomials of appropriate degree, yielding the normal flatness condition. Moreover, the parameters $\ut$ induce an isomorphism of the normal cone $N_{\oV/V}=\Spec_\oV(\gr_\cN(\cO_V))$ with the constant tube $\oV\times T_\ud$, where $T_\ud=\Spec(\QQ[t_1\.t_n]/(\ut^{I_\ud}))$ is the tubular thickening of $\Spec(\QQ)$ of width $\ud$.

(iii) We warn the reader that condition (1) does not follow from an isomorphism $N_{\oV/V}=\oV\times T_\ud$. Existence of appropriate nilpotents in $\cO_V$ is much subtler. In particular, we will see that it is not easy to prove some descent results for this property.

(iv) If $\oV\into V$ admits a retract $V\to\oV$ (e.g. this is the case when $V$ is a variety or a formal variety), then $V$ itself is isomorphic to the constant tube over $\oV$. This condition is still much stronger than just requiring that the normal cone is a constant tube.

(v) By our assumption $d_1\ge 2$ and hence the image of $\ut$ is a basis of the conormal bundle $T^*_{\oV/V}=\cN/\cN^2$ to $\oV$ (our definition of tubes rules out the possibility of an arbitrary width $\ud\in\Mord$ because any $i$ with $d_i=1$ would give rise to a dummy parameter $t_i=0$). It follows that if $s_1\.s_m\in\cO_{V,x}$ is a lifting of a family of regular parameters of $\cO_{\oV,x}$, then the image of $t_1\.t_n,s_1\.s_m$ in the cotangent space $m_{V,x}/m^2_{V,x}$ to $V$ at $x$ is a basis.

(vi) Assume that $V'\to V$ is either a regular morphism or the completion morphism $V'=\Spec(\hatcO_{V,x})\to V$ at $x\in V$. Then $V'$ is a tube, $\mord(V')=\mord(V)$ and the valuative filtration on $\cO_V$ pulls back to the valuative filtration on $\cO_{V'}$, that is, $\cF_{V'}=\cF_V\cO_{V'}$. Indeed, the claim reduces to the trivial check that if $V$ is a split tube with nilpotent coordinates $\ut$, then the same is true for $V'$ and the pullback of $\ut$ to $\cO_{V'}$.
\end{rem}

\subsection{Presentations of tubes in schemes}\label{pressec}
By a {\em tube} in a scheme $Z$ we mean a closed subscheme $V$ which is a tube and has a constant codimension in $Z$. This forces us to make certain equidimensionality assumptions and since we want everything to be compatible with formal completions, from this point and until the end of the paper we blindly assume that $Z$ is formally equidimensional even though this could be weakened at some points (but not in the main statements). We will now study ideals that define tubes and show, in particular, that tubes of a regular $Z$ correspond to roundings of integral centers.

\subsubsection{Formally equidimensional schemes}
Recall that a scheme $Z$ is {\em formally equidimensional} if for any $z\in Z$ the formal completion $\hatZ_z$ is equidimensional. For example, any normal G-scheme is formally equidimensional. For qe schemes this condition just means that $Z$ is excellent and locally equidimensional, see \S\ref{catenarysec}.

\subsubsection{Minimal embeddings}\label{minemb}
For a local ring $(A,m)$ let $e_A=\dim_k(m/m^2)$ denote its {\em embedding dimension} (or dimension of the cotangent space). Assume that $A$ is complete. Then any choice of the field of coefficients $k\into A$ and elements $t_1\.t_e$ whose images form a basis of $m/m^2$ induce a surjection $B=k\llbracket t_1\.t_e\rrbracket\onto A$ and $Z=\Spec(A)\into X=\Spec(B)$ is the minimal closed immersion of $Z$ into a regular complete local variety, which is unique up to an isomorphism. So, we call $Z\into X$ the {\em minimal embedding} of $Z$. Note that $m_B/m_B^2=m/m^2$. Of course, the minimal embedding is of constant codimension if and only if $Z$ is equidimensional.

\subsubsection{Tight tubes}
We say that a closed immersion $Y\into Z$ is {\em tight} at $z\in Y$ if the inequality of the embedding dimensions $e_{Y,z}\le e_{Z,z}$ is an equality. Equivalently, the surjective map of the cotangent spaces $T^*_{Z,z}\onto T^*_{Y,z}$ is an isomorphism. A tube $V\into Z$ is {\em tight} if it is tight at all points. In this case, the map $T^*_{\oV/Z}\onto T^*_{\oV/V}$ of the conormal bundles to $\oV$ is an isomorphism. We will later see that canonical tubes are tight, but for now we consider non-tight tubes too.

\subsubsection{Presentations}
By a {\em presentation} of a tube $V$ of width $\ud=(d_1\.d_n)$ in $Z$ we mean a presentation of the form $V=V_Z(\us,\ut^{I_\ud})$, where $\us,\ut\subset\cO_Z$ and the image of $\ut$ in $\cO_V$ is a family of nilpotent parameters of $V$. In particular, $V\into Z$ possesses presentations locally, and a global presentation exists if and only if $V$ is split.

\begin{rem}
A lift $\ut$ of the family of nilpotent parameters does not have to play any role in generating the ideal $\cI=(\us,\ut^{I_\ud})\subset\cO_Z$ corresponding to $V$. In fact, it can freely happen that $(\ut^{I_\ud})=0$ in $\cO_Z$. However, $\ut$ will be used to describe embeddings and to define tubular blowings up.
\end{rem}

\subsubsection{Tight presentations}
We say that a presentation of a tube $V$ in $Z$ is {\em tight}, if it does not include $s$-generators, that is, $V=V_Z(\ut^{I_\ud})$.

\begin{lem}\label{tubelem0}
Let $Z$ be a scheme, $V$ a split tube of width $\ud$ in $Z$ given by an ideal $\cI\subset\cO_Z$ and $\ut\subset\cO_Z$ a tuple mapped to a family of nilpotent parameters of $V$. Then $\ut^{I_\ud}\subset\cI$ and a tuple $\us\subset\cI$ gives rise to a presentation $\cI=(\us,\ut^{I_\ud})$ of $V$ in $Z$ if and only if the image of $\us$ generates $K=\Ker(T^*_{\oV/Z}\onto T^*_{\oV/V})$. In particular, locally at $z\in V$ the minimal size of $\us$ in a presentation of $V$ equals $\dim_{k(z)}(K)=e_{Z,z}-e_{V,z}$, and $V$ possesses a tight presentation $V=V_Z(\ut^{I_\ud})$ if and only if $V$ is tight.
\end{lem}
\begin{proof}
The set $\ut^{I_\ud}$ lies in $\cI$ because its image in $\cO_V$ vanishes. In addition, its image in $T^*_{\oV/Z}$ vanishes because $d_i>1$, and hence for any presentation we have that the image of $\us$ generates $K$. Conversely, assume that $K$ is generated by the image of $\us\subset\cI$ and let us prove that $\cI=(\us,\ut^{I_\ud})$. Since $V$ is a tight tube in $Z'=V_Z(\us)$, it suffices to prove that $\ut$ gives rise to a tight presentation of $V$ in $Z'$. So, we are reduced to the tight case and to simplify notation we assume in the sequel that $V$ is tight in $Z$.

With the new assumption we should prove that the closed immersion $V\into V'=V_Z(\ut^{I_\ud})$ is an isomorphism. This can be checked after the formal completion at $z\in V$, hence we can assume that $Z$ and $V$ are spectra of complete local rings. Then $Z$ possesses a minimal embedding $Z\into X=\Spec(A)$ and $\ut$ lifts to a partial family of regular parameters $\ut'\subset A$. Since $V'\into V''=V_X(\ut'^{I_\ud})$, it suffices to prove that $V=V''$. The closed immersion of regular local schemes $\oV\into V_X(\ut)$ is an isomorphism because the dimensions of both cotangent spaces are equal to $e_{Z,x}-n=\dim(X)-n$, where $n$ is the length of $\ud$. Recall that the rank of the free $\cO_\oV$-module $\gr_\cN(\cO_V)$ equals the number of elements in the set $\NN^n\setminus I_\ud$, and clearly $\gr_{\cN''}(\cO_{V''})$ is a free $\cO_\oV$-module of the same dimension (with a basis given by the images of monomials $\ut'^\ua$ with $\ua\in\NN^n\setminus I_\ud$). Therefore, the surjection $\cO_{V''}\onto\cO_V$ is an isomorphism.
\end{proof}

\subsubsection{Relation to integral centers}
In the particular case when $Z$ is regular we can interpret tubes as integral centers.

\begin{lem}\label{tubelem}
Let $X$ be a regular scheme, then

(i) There is a natural bijection between tubes of $X$ and integral centers $\cJ$, and opposite bijections are given by associating to a tube $V=V_X(\cI_V)$ the canonical center $\cJ=\cJ(\cI_V)$ and associating to a center $\cJ$ the vanishing locus $V=V_X(\cI)$ of its rounding $\cI=(\cJ)$. In particular, a closed subscheme $V$ is a tube if and only if its ideal is the rounding of an integral center.

(ii) Given a closed subscheme $Z\into X$ that corresponds to an ideal $\cI_Z$, the correspondence from (i) induces the natural bijection between tubes of $Z$ and $\cI_Z$-admissible integral centers of $X$.

(iii) If $V$ corresponds to $\cJ$ in (i), then $\mord(V)=\mord(\cJ)_{>1}$.
\end{lem}
\begin{proof}
It suffices to prove the claims locally, so we can freely assume that the tubes are split and the centers possess global presentations. Assume first that $\cJ$ is a center, choose a presentation and denote it $\cJ=[\us,\ut^\ud]$, where $\ud\in\Mord_{>1}$ and $s_1\.s_m,t_1\.t_n$ is a family of regular parameters along $V(\cJ)$. Then $\cI=(\cJ)$ equals $(\us,\ut^{I_\ud})$ and hence $V=V(\cI)$ is a tube of width $\ud$ and the image of $\ut$ is a family of nilpotent parameters of $V$. The canonical center of $\cI$ is $\cJ$ itself by Theorem~\ref{maxcenterth}.

Conversely, assume that $V$ is a tube of width $\ud=(d_1\.d_n)$ in $X$. By Lemma~\ref{tubelem0} locally one can present $V$ as $V=V_X(s_1\.s_m,\ut^{I_\ud})$, where $m=\dim(X)-\dim(\oV)-n$. Since $\oV=V_X(\us,\ut)$ is of codimension $m+n$ in $X$, the tuple $(\us,\ut)$ is a regular family of parameters along $\oV$, and therefore $\cJ=[\us,\ut^\ud]$ is an integral center such that $(\cJ)=(\us,\ut^{I_\ud})$ is the ideal of $V$ and $\mord(\cJ)=(1\.1,\ud)$.

The above correspondences establishes (i) and (iii). The claim (ii) follows too because $\cI_Z\subseteq\cJ$ if and only if $\cI_Z\subseteq(\cJ)$, which means that $V\into Z$.
\end{proof}

\begin{rem}\label{codimrem}
The lemma provides a justification for the notion of tubes we have introduced. In particular, it clarifies the requirement that a tube has constant codimension in $Z$ -- otherwise it would not define a center when $Z$ is regular, since the multiorder is constant along a center. Furthermore, for this reason we will have to only consider closed embeddings of constant codimensions $Z\into X$ into regular schemes, and this will restrict non-embedded resolution to formally equidimensional schemes.
\end{rem}

\subsubsection{Independence of coordinates}
The above results provide a useful tool for deducing various properties of tubes from those of weighted centers -- reduce by descent to the case of a complete local ring and then use a minimal embedding and the interpretation as a center. First of all, let us prove independence of coordinates of basic invariants.

\begin{cor}\label{tubecor}
Let $V$ be a tube, then

(i) Width and valuative filtration of $V$ are independent of choices of nilpotent parameters (in the split case), and hence are unique and completely determined by the underlying scheme.

(ii) Let $\ud=(d_1\.d_n)=\mord(V)$ and let $\oV$ denote the reduction of $V$. Then elements $t_1\.t_n\in\cO_V$ are nilpotent parameters of $V$ if and only if their images form a basis of $T^*_{\oV/V}$ over $\cO_\oV$ and $t_i\in\cF_{d_i^{-1}}$ for $1\le i\le n$.
\end{cor}
\begin{proof}
Let us assume first that $V=\Spec(A)$ for a complete local ring $A$, and choose a minimal embedding $V\into X$. Then by Lemma~\ref{tubelem} $V$ corresponds to a center $\cJ$ on $X$, and its width and valuative filtration are determined/induced by the multiorder and the valuation function of $\cJ$. In addition, the condition on $\ut$ in (ii) is equivalent to the condition that a lift $\ut'\in A$ of $\ut$ is a family of regular parameters along $\oV$ and $\nu_\cJ(t_i)\ge d_i^{-1}$. By \cite[Corollary~2.3.11(ii)]{dream_derivations} the latter happens if and only if $\cJ=[\ut'^\ud]$, which is equivalent to $\ut$ being a family of nilpotent parameters of $V$.

Consider now the general case. It suffices to check both claims locally, so we can assume that $V=\Spec(A)$ for a local ring $A$. If (i) fails and the width or the valuation filtration differ for two different choices of nilpotent parameters, then the same is true after the completion, contradicting the case we have already established. To prove (ii) we note that the direct implication is obvious, so let us prove the opposite one. The inclusions $t_i\in\cF_{d_i^{-1}}$ imply that $\ut^{I_\ud}=0$. In addition, the images of $\ut^\ua$ with $\ua\in\NN^d\setminus I_d$ form a basis of $\gr_\cN(A)$ over $B=A/\cN$ because their images form a basis of $\gr_{\widehat\cN}(\hatA)$ over $\hatB=\hatA/\hatN$ by the already established case of $\hatV$, and one has that $$\gr_\cN(A)\otimes_B\hatB=\gr_{\widehat\cN}(\hatA).$$
\end{proof}

\subsubsection{Stacky tubes}
As another consequence we can now prove that being a tube is a smooth-local property, and hence the notion of tubes naturally extend to stacks: a stack $V$ is called a tube if it possesses a smooth presentation $V_0\to V$ with $V_0$ a tube.

\begin{cor}\label{tubesmooth}
Assume that  $V'$ is a tube of width $\ud$ and $f\:V'\to V$ is a surjective smooth morphism. Then $V$ is a tube of width $\ud$.
\end{cor}
\begin{proof}
Note that $V''=V'\times_VV'$ is a tube and by Corollary \ref{tubecor}(i) the valuative filtration $\cF'$ on $V'$ pulls backs to that of $V''$ under both projections $V''\to V'$. By flat descent, $\cF'$ is the pullback of a filtration $\cF$ on $\cO_V$. Also, by the flatness of $V'\to V$, the filtration induced by $\cF$ on $T^*_{\oV/V}=\cN/\cN^2$ pulls back to the filtration $\cF'$ induces on $T^*_{\oV'/V'}$. Any family of nilpotent parameters $\ut'$ on $V'$ is mapped to a basis of $T^*_{\oV'/V'}$ and satisfies $t'_i\in\cF'_{d_i^{-1}}$ for $1\le i \le n$. Therefore there exists a tuple $\ut$ of nilpotents on $V$ which is mapped to a basis of $T^*_{\oV/V}$ and satisfies $t_i\in\cF_{d_i^{-1}}$ for $1\le i \le n$. By Corollary~\ref{tubecor}(ii) the pullback of $\ut$ is a family of nilpotent parameters on $V'$, hence $\ut^{I_\ud}=0$ already on $V$. This implies that $V$ is a tube with $\ut$ a family of nilpotent parameters.
\end{proof}

\subsubsection{Regular descent}
To establish algebraization we will also need the stronger descent result when $V'\to V$ is regular, or at least a completion morphism of a $G$-scheme. In this case, the fiber product is non-noetherian, so another argument is needed. The goal is to construct nilpotent parameters downstairs. The conditions $\ut^\ua=0$ with $\ua\in I_\ud$ are not that easy to achieve and the argument should use the $G$-condition. One way to do this is via Artin's approximation. We prefer instead of this to consider the general case when $V'\to V$ is regular and use Popescu's theorem. It seems certain that the latter can be avoided by an argument imitating the work with integral closures in the case of centers and roundings, but this would probably require a few more pages of technical work.

\begin{theor}\label{tubedescent}
Assume that $V'$ is a tube of width $\ud$ and $f\:V'\to V$ is a surjective regular morphism. Then $V$ is a tube of width $\ud$.
\end{theor}
\begin{proof}
The filtrations induced on $\cO_V$ and $\cO_{V'}$ by the nilradicals are compatible by the regularity of $f$. Since $V'$ has regular reduction $\oV'$ along which it is normally flat, the same is true for $V$. We can work locally on $V$, hence we can assume that $\gr_\cN(\cO_V)$ is a free $\cO_V$-module. Computing on $V'$ we see that its rank $r$ equals the cardinality of the set $\NN^n\setminus I_\ud$.

By Popescu's theorem, see \cite{Popescu} or \cite{Spivakovsky}, $V'$ is the filtered limit of smooth $V$-schemes $V_i$. Let $\cN_i$ and $\oV_i$ denote the nilradical and the reduction of $V_i$, respectively. Then $\gr_{\cN_i}(\cO_{V_i})$ is a free $\cO_{\oV_i}$-module of rank $r$ by the regularity of $V_i$ over $V$. Let $\ud=\mord(V')$ and let $t'_1\.t'_n$ be a family of nilpotent parameters of $V'$. Then $\ut'$ comes from a tuple $\ut\in\cO_{V_i}$ for a large enough $i$. Moreover, enlarging $i$ we can achieve that $\ut$ satisfy the same relations as $\ut'$, namely, $\ut^\ua=0$ for any $\ua\in I_\ud$. In the same manner, enlarging $i$ we can assume that the image of $\ut$ is a basis of $\cN_i/\cN_i^2$, and hence $\cN_i=(\ut)$. Finally, the images of the monomial $\ut^\ua$ with $\ua\in \NN^n\setminus I_\ud$ are linearly independent in $\gr_{\cN_i}(\cO_{V_i})$ because their images in $\gr_{\cN}(\cO_{V})$ are linearly independent. Thus, $V_i$ is a tube and it remains to use Corollary~\ref{tubesmooth}.
\end{proof}

\subsection{Canonical tubes}
The notion of a (formally) canonical tube of a formally equidimensional scheme $Z$ is defined precisely as the notion of a (formally) canonical center of an ideal $\cI$ on a regular scheme $X$, see \cite[\S3.3]{dream_derivations} and \S\ref{formcan}, with the only subtlety that we require the tube to be tight. Of course, in the case of $Z=V_X(\cI)$ the definitions are equivalent due to the dictionary of Lemma~\ref{tubelem}(i), and the tightness assumption is needed to control the initial ones in the multiorder. For the sake of completeness we spell out the definitions quickly.

\subsubsection{The invariant}\label{tubeinvar}
For a tight tube $V$ in $Z$ of width $\ud=(d_1\.d_n)=\mord(V)$ we define the {\em invariant} of $V$ in $Z$ to be $\inv(V)=(\ud,|V|)$ ordered lexicographically. As with invariants of centers, the last entry is a closed set and the partial order on closed sets is given by inclusion.

\subsubsection{Canonical stratification}
Let $Z$ be a scheme (or a stack). A tube $V$ of $Z$ is called {\em maximal} if it is tight and has maximal possible invariant: any other tight tube $V'\neq V$ in $Z$ satisfies $\inv(V')<\inv(V)$. If, furthermore, for any open $U\subseteq Z$ with $V\cap U\neq\emptyset$ one has that $V_U=V\times_ZU$ is the maximal tube of $U$, then $V$ is called the {\em canonical tube} of $Z$. Finally, we say that $Z$ possesses a {\em canonical stratification} if for any point $z\in Z$ there exists a neighborhood $U$ which possesses a canonical tube $V_U$ and $z\in V_U$. In this case the upper semicontinuous {\em width function} $\mord_Z\:Z\to\Mord$ arises.

\subsubsection{The embedded case}
If $Z$ embeds into a regular scheme, then using the correspondence between tubes and centers established in Lemma~\ref{tubelem} we can reduce studying canonical tubes to the already established theory of canonical centers.

\begin{lem}\label{embedded}
Let $X$ be a regular scheme and $Z=V_X(\cI)$ a closed subscheme of constant codimension, then

(i) A tube $V$ of $Z$ is canonical if and only if the corresponding integral $\cI$-admissible center $\cJ$ on $X$ is canonical.

(ii) If $X$ is excellent, then the canonical stratification of $\cI$ and the multiorder function $\mord_\cI$ induce the canonical stratification by tubes and the width function on $Z$.
\end{lem}
\begin{proof}
Let $\cJ$ be an $\cI$-admissible center and let $V$ be the corresponding tube (the vanishing locus of $(\cJ)$). Assume that $\cJ$ is $\cI$-canonical, and let us deduce that $V$ is tight. We can work locally at $x\in V$. Let $n=\dim_x(X)$ and $\mord(\cJ)=(1\. 1,\ud)$, where $\ud\in\Mord_{>1}$ and the multiorder starts with $c$ ones. Then it is easy to see that $e_{V,x}=n-c$. In addition, after shrinking $X$ around $x$ we can assume that there exists a presentation $\cJ=[s_1\.s_c,\ut^\ud]$, and hence $\us\subset\cI$ and $Z\into V_X(\us)$. Thus, $e_{V,x}\le e_{Z,x}\le n-c=e_{V,x}$ and all inequalities are equalities, which means that $V$ is tight in $Z$ at $x$.

We see that in proving either direction in (i) one can assume that $V$ is tight, and then the assertion follows from Lemma~\ref{tubelem}. Furthermore, if $X$ is excellent, then $\cI$ possesses a canonical stratification by Theorem~\ref{canonicalcenter}, so (ii) follows from (i).
\end{proof}

\subsubsection{Formal canonicity}
Now we can proceed similarly to \S\ref{algsec}. First, for any formally equidimensional scheme $Z$ and a point $z\in Z$ the formal completion $\hatZ_z$ embeds into a regular scheme with constant codimension and hence possesses a canonical tube $\hatV_z$. We define the {\em formal width function} $\ford_Z\:Z\to\Mord$ by $\ford_Z(z)=\mord(\hatV_z)$. Also, a tube $V$ of $Z$ is called {\em formally canonical} if $|V|$ consists of all points where $\ford_Z$ attains the maximal value and $V\times_Z\hatZ_z=\hatV_z$ for any $z\in V$.

\begin{lem}\label{formaltube}
Let $Z$ be a formally equidimensional scheme with a tube $V$, and let $f\:Z'\to Z$ be a regular morphism and $V'=V\times_ZZ'$ the pullback.

(i) If $f$ is surjective and $V'$ is canonical, then $V$ is canonical.

(ii) If $V$ is formally canonical, then $V$ is canonical.

(iii) If $V$ is formally canonical, then $V'$ is formally canonical, and the opposite implication holds whenever $f$ is surjective.
\end{lem}
\begin{proof}
For a point $z'\in Z'$ with $z=f(z')$ one has that $e_{Z',z'}=e_{Z,z}+\dim_{z'}(Z'_z)$ and the same is true for $V$. Therefore, if $V$ is tight in $Z$, then $V'$ is tight in $Z'$, and the opposite is true when $f$ is surjective. The remaining arguments are analogous to the proof of Lemmas~\ref{formalcenter}, so we skip the details.
\end{proof}

\subsubsection{Descent of canonical centers}
Now, we can establish the main result about descent of formally canonical centers in the local case.

\begin{lem}\label{cantubedescent}
Assume that $Z$ is a formally equidimensional local G-scheme with the closed point $z$ and $U=Z\setminus \{z\}$ possesses a formally canonical tube $V_U$. Then $Z$ possesses a formally canonical tube.
\end{lem}
\begin{proof}
Again, this is proved precisely as its analog in Lemma~\ref{formaldescent}, so we just outline the argument. The canonical tube $\hatV$ of $\hatZ=\Spec(\hatcO_z)$ exists by Lemma~\ref{embedded}(ii), and it is enough to prove that it is the pullback of a closed subscheme $V$ of $Z$ because then $V$ is a tube by Theorem~\ref{tubedescent} and this tube is formally canonical by Lemma~\ref{formaltube}.

If $\hatV$ is supported at the preimage of $z$, then it is given by an open ideal and hence comes from a scheme $V\into Z$ supported at $z$. Otherwise we take $V$ to be the schematic closure of $V_U$ and use that tubes have no embedded components and schematic closures are compatible with formal completions.
\end{proof}

\subsection{Spreading out canonical tubes}
It remains to show that if a tube is canonical locally at a point on an N-scheme, then the same is true in a neighborhood. As in the case of centers we will have to use Tschirnhaus coordinates.

\subsubsection{Leading term of a tube}
If $V$ is a tube given by an ideal $\cJ\subset\cO_Z$ and $\ocJ$ is the radical of $\cJ$, set $\cJ_{>1}=\ocJ\cJ$ and define the {\em leading term} of $V$ in $Z$ to be the $\cO_\oV$-module $\cT_\cJ=\cJ/\cJ_{>1}$. Informally, it can be viewed as an additional graded piece $\cF_{>1}/\cF_{\ge 1}$ of the valuative filtration, and this graded piece already depends on the embedding in $Z$ and not only on $V$.

%Assume that $\cJ$ possesses a presentation $(\us,\ut^{I_\ud})$ and let $\cJ_{>1}$ be the ideal generated by monomials $\us^\ub\ut^\ua$ such that $\sum_i b_i+\sum_j\frac{a_j}{d_j}>1$. By the we mean the $\cO_\oV$-module $\cT_\cJ=\cJ/\cJ_{>1}$.

\begin{lem}\label{conetube}
Let $Z$ be a formally equidimensional scheme and let $V=V(\cJ)$ be a tight tube with a presentation $\cJ=(\ut^{I_\ud})$. Then
$\cT_\cJ$ is generated by the images of monomials $\ut^\ua$ with $\ua\cdot\ud^{-1}=1$.
\end{lem}
\begin{proof}
The claim can be checked after the completion at a point, hence we are reduced to the case when $Z=\Spec(A)$ with a complete local $A$. Then there exists a minimal embedding $Z\into X$ and $V$ is given by a center $\cJ'=[\ut'^\ud]$ in $X$, where $\ut'$ is a lift of $\ut$ to $X$. It remains to note that $\cT_\cJ$ is the quotient of $\cT_{\cJ'}$ by the image $t_{\cJ'}(\cI)$ of the ideal $\cI\subset\cO_X$ of $Z$, hence the claim follows from Lemma~\ref{conebase}.
\end{proof}

\subsubsection{Kernel of the leading term}
In view of Lemma \ref{conetube} by the {\em kernel} of the leading term of a tube with a presentation $V=V_Z(\ut^{I_\ud})$ we mean the kernel of the map $F\to\cT_\cJ$, where $F$ is the free $\cO_V$-module spanned by monomials $\ut^\ua$ with $\ua\cdot\ud^{-1}=1$. Moreover, this notion globalizes to arbitrary tubes of $Z$ because the local definition is independent of the presentation. Indeed, one can pass to formal completions, use a minimal embedding $Z\into X$ and then $K$ is nothing else but the leading term $t_{\cJ'}(\cI)$ of the ideal $\cI$ defining $Z$ in $X$.

\subsubsection{Tschirnhaus coordinates}
Let $V$ be a tube of $Z$ of width $(d_1\.d_n)$ and let $\cJ$ be the ideal defining $V$. By a {\em Tschirnhaus presentation} of $V$ in $Z$ we mean a tight presentation $\cJ=(\ut^{I_\ud})$, where $t_1\.t_n\in\cO_Z$ satisfy the following conditions: for any $1\le i\le n$ there exist $c_1\.c_i\in\NN$ and an element $f\in K$ in the kernel of $\cT_\cJ$ of the form $f=\sum_{\ua} b_\ua\ut^\ua$ with $\ua\cdot\ud^{-1}=1$ such that

\begin{itemize}
\item[(i1)] there exists $\ua=(c_1\.c_i,0\.0)$ with $b_\ua=1$ and $c_i\neq 0$,
\item[(i2)] there exists no $\ua$ starting with $(c_1\.c_{i-1},c_i-1)$ and such that $b_\ua\ne 0$.
\end{itemize}

Similarly to Theorem \ref{tschirnhausexist} the following result holds:

\begin{theor}\label{tubetschirnhaus}
The following conditions on a tube $V$ of a local formally equidimensional G-scheme $Z$ are equivalent:

(i) $V$ is formally canonical,

(ii) $V$ is canonical,

(iii) $V$ possesses a Tschirnhaus presentation (in particular, $V$ is tight)
\end{theor}
\begin{proof}
We will act similarly to the proof of Theorem \ref{tschirnhausexist}, and will also reduce some part of the proof to that theorem. By Lemma~\ref{formaltube}(ii) we know that (i) implies (ii). Assume that $V$ is canonical and let us construct a Tschirnhaus presentation by modifying a tight presentation $\cJ=(\ut^\ud)$. By induction, we can assume that the two Tschirnhaus conditions are satisfied for $1\.i-1$, and we should prove that replacing $t_i\.t_n$ by $t'_i\.t'_n$ one can obtain another presentation which satisfies also (i1) and (i2). %The canonicity assumption implies that subscheme $V(\ut^{\ud'})$ with $d'_j=d_j$ for $j<i$ and $d'_j>d_i$ for $j\ge i$ is not a tube.

%For $\veps\ge 0$ let $\ud'=\ud_\veps$ be defined by $d'_j=d_j$ for $1\le j\le i-1$ and $d'_j=d_i+\veps$ for $i\le j\le n$.
Consider the kernel $K=\Ker(F\to\cT_\cJ)$ of the leading term, where $F$ is the free $\cO_\oV$-module with basis formed by $\ut^\ua$ with $\ua\cdot\ud^{-1}=1$. We claim that $K$ contains an element of the form $f=\sum_\ua b_\ua\ut^\ua$ involving a non-zero term $b_\uc t^\uc$ such that $\uc\cdot\ud_0^{-1}=1$ and $c_j\neq 0$ for some $j\ge i$. To prove this it suffices to show that if the claim fails, then there exists $\ud'=(d'_1\.d'_n)>\ud$ such that $V'=V_Z(\ut^{I_{\ud'}})$ is also a tube, contradicting the maximality of $V$. Note that the coordinates of the presentation are given and we only play with the exponents, hence the claim can be checked after formal completion. In this case, $Z$ possesses a minimal embedding $Z\into X$ into a regular scheme, and then $K=t_{\cJ'}(\cI)$, where $\cJ'=[\ut'^\ud]$ for a lift $\ut'\subset\cO_X$ of $\ut$ and $\cI$ is the ideal of $Z$. It was observed in the second paragraph of the proof of Theorem~\ref{tschirnhausexist} that if $t_{\cJ'}(\cI)$ does not contain such an element $f$ as above, then there exists an $\cI$-admissible center $\cJ'_\veps=[\ut'^{\ud_\veps}]$ such that $\ud_\veps>\ud$. This center does not have to be integral, so its rounding $(\cJ'_\veps)=(\ut'^{I_{\ud_\veps}})$ does not have to be a tube. However, the canonical center $\cJ''$ of $(\cJ'_\veps)$ is an integral center, and it is easy to see that it has a presentation $\cJ''=[\ut'^{\ud'}]$ with the same coordinates $\ut'$ and $\ud'\ge\ud_\veps$ (the equality holds if and only if $\cJ'_\veps=\cJ''$ is itself integral). For example, one can use the same algebraization and torus equivariance argument as in the proof of Theorem~\ref{maxcenterth}. Thus, $V=(\ut^{I_{\ud'}})$ is a tube in $Z$ with a larger invariant.

Once we have an appropriate element $f\in K$, we act precisely as in the proof of \ref{tschirnhausexist}. First, renormalize $f$ so that $b_\uc=1$. Then note that the monomial only involves $t_j$ with $d_j=d_i$ and use a general enough substitution $t_j+b_it_i$ with $b_i\in\QQ$ and subsequent renormalization to achieve that $f$ contains a monomial $xt_i^e$, where $x$ is a monomial in $t_1\.t_{i-1}$. After this step in addition to the two Tschirnhaus conditions for $1\.i-1$, the condition (i1) holds. Finally, the appropriate Tschirnhaus substitution $t'_i=t_i-e^{-1}\sum_{j=1}^l xy_j$ preserves these conditions and also achieves that (i2) is satisfies.

(iii)$\Longrightarrow$(i) Assume now that $\cJ=(\ut^{I_\ud})$ is a Tschirnhaus presentation and let us prove that $V$ is formally canonical. Passing to completions preserves the Tschirnhaus presentation, hence we can assume that $Z=\Spec(A)$, where $B$ is a complete local ring. Choose a minimal embedding $Z\into X=\Spec(B)$ and let $\ut'\in A$ be liftings of $\ut$. Then $\cJ'=[\ut'^\ud]$ is Tschirnhaus presentation too, and hence $\cJ'$ is a canonical center by Theorem~\ref{tschirnhausexist}. So, $V$ is a canonical tube of $Z$ by Lemma~\ref{embedded}(i).
\end{proof}

As in the case of canonical centers and Corollary \ref{extendcor}, the Tschirnhaus conditions are open: if a tube $V$ possesses a Tschirnhaus presentation at a point $z$, then the same presentation is Tschirnhaus for any point in a neighborhood of $z$. Therefore Theorem~\ref{tubetschirnhaus} implies the following result:

\begin{cor}\label{extendcor1}
Assume that $Z$ is a formally equidimensional excellent scheme, $V$ a closed subscheme and $z\in Z$ a point such that the localization $V_z$ is the canonical tube of $Z_z$. Then there exists a neighborhood $U$ of $z$ such that $V_U=V|_U$ is the canonical tube of $U$.
\end{cor}

\subsection{Tubular blowings up}
The last ingredient we need for non-embedded resolutions is an analogue of the normalized root blowing up. Of course in the embedded case this should be compatible with the normalized root blowing up of a regular ambient scheme. It is not clear if such an operation can be defined by a natural universal property, so we will give an explicit definition in terms of a presentation.

\subsubsection{Tubular Rees algebras}
As in \cite[\S3]{Quek-Rydh} a Rees algebra on $Z$ is a finitely generated graded $\cO_X$-algebra $\cR=\oplus_{d\in\NN}\cI_dt^d\subseteq\cO_X[t]$ such that $\cI_d\supseteq\cI_{d+1}$ for any $d$. Given a scheme $Z$ with a tight tube $V$ and a number $N\ge 1$, we define {\em tubular blowing up} $Z'=\Bl_{\sqrt[N]V}(Z)$ as follows. If $V$ possesses a tight presentation $V=V_Z(\ut^{I_\ud})$, then the {\em tubular Rees algebra} $\cR=\cR_{\sqrtNV}\subset\cO_X[t]$ is defined as follows: the graded piece $\cR_n$ is generated by monomials $\ut^\ua$ with $N\sum_i\frac{a_j}{d_j}\le n$. Unlike fractional powers of weighted centers, $\sqrtNV$ is only a notation, but it indicates the analogy.

\begin{lem}\label{tuberees}
Let $Z$ be a formally equidimensional scheme, $V$ its tube and $N\ge 1$ a natural number.

(i) If $V$ possesses a presentation, then $\cR_{\sqrtNV}$ is independent of the choice of a presentation, and hence the notion of tubular Rees algebras globalizes to the general case.

(ii) If $X$ is regular, $Z\into X$ is a tight closed immersion and $\cJ$ is the center of $X$ whose rounding defines $V$, then the tubular Rees algebra $\cR_\sqrtNV$ is the restriction to $Z$ of the Rees algebra $\cR_\sqrtNJ$ of the center $\sqrtNJ$.
\end{lem}
\begin{proof}
(ii) Assume first that $V$ possesses a presentation $V=V_Z(\ut^{I_\ud})$. Then it lifts to a presentation $\cJ=[\ut'^\ud]$, and the explicit description of the Rees algebras of centers implies the claim. In particular, $\cR_\sqrtNJ$ is determined by $\cJ$ and $N$ (or just by $\sqrtNJ$) and does not depend on the presentation. The global case follows.

(i) The claim can be checked formally locally, in which case $Z$ admits a minimal closed immersion $Z\into X$ into a local regular scheme, and the claim follows from (ii).
\end{proof}

\subsubsection{Tubular blowings up and relation to weighted blowings up}
Recall that the {\em blowing up} along a Rees algebra $\cR$ is defined to be the stacky proj $$\Bl_\cR(X)=\cProj_X(\cR):=[(\Spec_X(\cR)\setminus V(\cR_{>0}))/\GG_m],$$ which means that one refines the usual proj construction by taking the stack theoretic quotient by the $\GG_m$-action.
Naturally, we define the {\em tubular blowing up} of $Z$ along $\sqrtNV$ to be the blowing up of the associated Rees algebra $\Bl_\sqrtNV(Z)=\cProj_Z(\cR_\sqrtNV)$.

\begin{lem}\label{tubeblow}
Let $Z$ be a formally equidimensional scheme, $V$ its tube, $N\ge 1$ a natural number and $Y=\Bl_\sqrtNV(Z)$.

(i) If $Z'\to Z$ is a regular morphism and $V'=V\times_ZZ'$, then $\Bl_{\sqrt[N]{V'}}(Z')=Y\times_ZZ'$.

(ii) If $Z\into X$ is a closed immersion with a regular target and $\cJ$ is the center on $X$ corresponding to $V\into X$, then the tubular blowing up $Y\to Z$ is the strict transform of the weighted blowing up $\Bl_{\cJ^{1/N}}(X)\to X$.
\end{lem}
\begin{proof}
The first claim follows from the functoriality of tubes for regular morphism, see Remark \ref{tuberem}(vi), and funtoriality of blowings up of Rees algebras with flat morphisms. In the second claim we use Lemma~\ref{tuberees}(ii) and \cite[Corollary~3.2.14]{Quek-Rydh}.
\end{proof}

\subsection{Main resolution results}
Finally, we can summarize all we have done in \S\ref{nonembsec} in the two main results: existence of tubular stratification and the resulting non-embedded resolution.

\subsubsection{The canonical stratification by tubes}
We have the following analogue of Theorem \ref{canonicalcenter} and the argument is the same up to replacing lemmas about centers by their analogs about tubes.

\begin{theor}\label{canonicaltube}
Let $Z$ be a noetherian locally equidimensional excellent stack of characteristic zero, then

(i) $Z$ possesses a canonical stratification by tubes $Z=\coprod_i V_i$. In particular, the width function $\mord_Z\:Z\to\Mord$ is well-defined and upper semicontinuous.

(ii) If $f\:Z'\to Z$ is a regular morphism, then $\coprod_iV_i\times_ZZ'$ is the canonical stratification of $Z'$ and $\mord_Z=\mord_{Z'}\circ f$.

(iii) Let $V$ be the canonical center with $(d_1\.d_n)=\mord_X(\cI)$, let $N>0$ be a natural number such that $N/d_i\in\NN$ for $1\le i\le n$ and let $Z'=\Bl_{\sqrt[N]V}(Z)\to Z$ be the associated tubular blowing up. Then $\mord(Z')<\mord(Z)$.
\end{theor}
\begin{proof}
(i) Existence of the canonical stratification implies the rest. By notherian induction, it suffices to show that if an open $W\subsetneq Z$ possesses a canonical stratification, then the same is true for a larger open. Choose any generic point $\eta$ of the complement $Z\setminus W$. Then $Z_\eta\setminus \eta$ lies in $W$ and hence possesses canonical stratification by the functoriality. Therefore, $\cI_\eta$ possesses a canonical tube by Lemma~\ref{cantubedescent}. The latter tube extends by Corollary~\ref{extendcor1} to the canonical tube $V_U$ of a neighborhood $V\subseteq Z$ of $\eta$, and the claim follows easily.

Claim (ii) follows from Lemma \ref{formaltube} and (iii) can be checked after the formal completion, when a minimal embedding exists, the width coincides with the multiorder and the latter drops by \cite[Theorem~3.4.4]{dream_derivations}.
\end{proof}

\subsubsection{Dream non-embedded resolution}\label{dreamressec}
Finally, we can define the {\em non-embedded dream resolution} of an arbitrary locally equidimensional reduced excellent scheme $Z$. This is an inductively defined sequence $\dots Z_1\to Z_0=Z$, where for each $i$ we set $Z_{i+1}=\Bl_{\sqrt[N_i]V_i}(Z_i)$ with $V_i$ the canonical tube of $Z_i$ and $N_i$ the smallest positive integer such that $N_j/d_j\in\NN$ for $1\le j\le n$ and $(d_1\.d_n)=\mord(V_i)$. The process stops when a regular $Z_n$ is achieved and so $V_n=Z_n$ has $\mord(V_n)=\emptyset$. Theorem~\ref{canonicaltube} immediately implies our main result on non-embedded resolution -- Theorem~\ref{resolutionth}.

\subsubsection{Destackification}
Destackification of stacks with diagonalizable stabilizers and its relative version are compatible with regular morphisms and the same is true for normalization of qe schemes, hence we obtain the following method to resolve qe schemes which is functorial with respect to regular morphisms.

\begin{cor}\label{resolvecor}
Let $Z$ be a reduced quasi-excellent stack of characteristic zero. Composing the normalization $Z_0=Z^\nor\to Z$, the non-embedded resolution $Z_n\to\dots\to Z_0$ and the relative $Z$-destackification $Z'\to Z_n$ one obtains a proper birational representable morphism $Z'\to Z$ with a regular source. The construction is compatible with arbitrary regular morphisms $Y\to Z$.
\end{cor}

\subsubsection{Other categories}
The analogues of tubes and centers on regular schemes can be straightforwardly defined in the categories of complex and non-archimedean analytic spaces and excellent formal schemes. Since canonical stratifications and dream principalization and resolution constructed in this paper hold for excellent schemes of characteristic zero and are functorial with respect to regular morphisms, the usual arguments imply that the analogues of Theorems~\ref{canonicalcenter}, \ref{principalizationth}, \ref{embeddedth}, \ref{canonicaltube} and \ref{resolutionth} hold in the categories of complex and non-archimedean analytic spaces and excellent formal schemes of characteristic zero.

\bibliographystyle{amsalpha}
\bibliography{canonical_center}

\end{document}